\documentclass[11pt, reqno]{amsart}

\usepackage[left=1in,top=1.3in,right=1in,bottom=1in,footskip = 0.333in]{geometry}

\usepackage{amsmath, amsfonts, amssymb, amsbsy, bbm, bigstrut, graphicx, enumerate,  upref, longtable, comment, comment,xcolor}
\usepackage{float}
\usepackage{tabularx, booktabs, makecell, caption}
\usepackage{siunitx}

\usepackage{subcaption}
\usepackage[breaklinks]{hyperref}

\usepackage[numbers, sort&compress]{natbib} 
\usepackage{comment}
\usepackage[export]{adjustbox}

\newtheorem{thm}{Theorem}
\newtheorem{lemma}[thm]{Lemma}
\newtheorem{corollary}[thm]{Corollary}

\theoremstyle{definition}
\newtheorem{remark}[thm]{Remark}

\newtheorem{ass}{Assumption}

\newcommand{\bx}{\mathbf{x}}

{\tiny {\tiny }}

\newcommand{\bz}{\mathbf{z}}

\newcommand{\bE}{ \mathbb{E} }

\newcommand{\bN}{ { N} }

\newcommand{\bQ}{ { Q} }

\newcommand{\bX}{ { X} }
\newcommand{\bY}{ { Y} }
\newcommand{\bZ}{ { Z} }

\newcommand{\bP}{\mathbb{P}}

\newcommand{\bzero}{ {\bf 0} }
\def \bX{\mathbf{X}}
\def \bQ{\mathbf{Q}}
\def \bmu{\pmb{\mu}}

\DeclareMathOperator{\sech}{sech}

\makeatletter
\def\thickhline{%
	\noalign{\ifnum0=`}\fi\hrule \@height \thickarrayrulewidth \futurelet
	\reserved@a\@xthickhline}
\def\@xthickhline{\ifx\reserved@a\thickhline
	\vskip\doublerulesep
	\vskip-\thickarrayrulewidth
	\fi
	\ifnum0=`{\fi}}
\makeatother

\newlength{\thickarrayrulewidth}
\numberwithin{equation}{section}

\renewcommand{\hat}{\widehat}
\renewcommand{\tilde}{\widetilde}

\begin{document}
\title[Nonsense association]{Nonsense associations in Markov random fields with pairwise dependence}

\author[S. Bhattacharya]{Sohom Bhattacharya}
\address{S. \ Bhattacharya\hfill\break
	Department of Statistics\\ University of Florida\\ Gainesville, FL 32611, USA.}
\email{bhattacharya.s@ufl.edu}
\author[R. Mukherjee]{Rajarshi Mukherjee}
\address{R. \ Mukherjee\hfill\break
	Departments of Biostatistics\\ Harvard T.H. Chan School of Public Health\\ Boston, MA 02115, USA.}
\email{ram521@mail.harvard.edu}
\author[E. Ogburn]{Elizabeth Ogburn}
\address{E. \ Ogburn\hfill\break
	Departments of Biostatistics\\ Johns Hopkins Bloomberg School of Public Health\\ Baltimore, MD, 21205, USA.}
\email{eogburn@jhu.edu}

\begin{abstract}
   Yule (1926) identified the issue of ``nonsense correlations" in time series data, where dependence within each of two random vectors causes overdispersion -- i.e. variance inflation -- for measures of dependence between the two. During the near century since then, much has been written about nonsense correlations -- but nearly all of it confined to the time series literature. In this paper we provide the first, to our knowledge, rigorous study of this phenomenon for more general forms of (positive) dependence, specifically for Markov random fields on lattices and graphs. We consider both binary and continuous random vectors and  three different measures of association: correlation, covariance, and the ordinary least squares coefficient that results from projecting one random vector onto the other. In some settings we find variance inflation consistent with Yule's nonsense correlation. However, surprisingly, we also find variance \textit{deflation} in some settings, and in others the variance is unchanged under dependence. Perhaps most notably, we find general conditions under which OLS inference that ignores dependence is valid despite positive dependence in the regression errors, contradicting the presentation of OLS in countless textbooks and courses.  
\end{abstract}

\maketitle

\section{Introduction}

\noindent It is well known that two independent, non-stationary time series often appear to be correlated with one another and that tests of independence have inflated type-I error \citep{phillips1986understanding,ernst2017yule}; this phenomenon is known variously as  \textit{nonsense}, \textit{spurious}, or \textit{volatile} correlations and has been the subject of extensive research. Less well-known is that related phenomena can occur in many other dependent data settings, e.g. dependence due to spatial, batch, population, genetic, or network structure, and that similar phenomena can occur for different measures of association.
We will use the term \textit{association} to be more general than \textit{correlation} and to encompass correlation coefficients, estimated covariances, and regression coefficients as measures of linear non-independence. 

The problem of nonsense correlations in time series was first articulated by \cite{yule1926we}, building upon earlier works by \cite{hooker1905correlation,student1914iv,yule1921time}. Much work has been on pre-whitening data, e.g. extracting independent increments, before attempting to learn about associations between time series. \cite{ernst2017yule} derived the exact asymptotic distribution of the estimated correlation of two independent Weiner processes, showing overdispersion that does not vanish asymptotically and ``solving" the problem first put forth by \cite{yule1926we}.

However, outside of the time series econometrics literature, this idea seems to be largely unknown. We have been able to find a smattering of references in the broader literature:  of \cite[Section 12.2.3]{kass2014analysis} illustrates nonsense associations due to temporal autocorrelation in linear regression. \cite{harris2020nonsense} discussed the importance of dealing with nonsense correlations in the study of neurophysiological processes in neuroscience. \cite{royama2005moran} noted that natural population processes (e.g. population growth) are susceptible to nonsense correlations. Moving beyond time series data, \cite{student1914iv} first noted that, in analogy to temporal processes, spatial processes are also susceptible to more highly variable estimates of correlation, and that the problem there can similarly be solved by extracting independent increments. \cite{clifford1989assessing} and \cite{dutilleul1993modifying} derived the effective sample size for tests of correlation between two spatial processes under strong parametric assumptions. Section 12.4 of \citealp{efron2016computer} includes a brief discussion of \textit{ephemeral predictions}, which is a manifestation of nonsense associations in the prediction setting.  
Recently \cite{lee2021network} made a case for the importance of nonsense associations in social network data. \cite{bhattacharya2022pc} studied the power and type I error of a family of tests of association commonly used in genome-wide association studies (GWAS), where the independent (e.g. phenotype) and dependent (e.g. SNP) variables both exhibit dependence due to family or population structure; they showed that standard tests of association can have type I error bounded away from the nominal level or even converging to 1.

In addition to statistical papers, a handful of papers from the social sciences have pointed out phenomena that appear to be instances of nonsense associations, though none of these papers make the connection explicitly: ``Galton's problem," a well-known concept in cross-cultural studies and anthropology, refers to the fact that within-culture dependence makes it difficult to perform inference using cross-cultural data, and at least one group of authors have connected this to network dependence \citep{dow1984galton}. \cite{dellaposta2015liberals} describes how incidental clustering of beliefs or behaviors in social networks can look like systematic associations. \cite{kelly2019standard} shows that many findings in the economics literature on ``persistence" or ``deep origins," purporting to find connections between characteristics of a place in modern day and characteristics of the same place in the distant past, can be explained by spatial autocorrelation rather than any meaningful association between past and present characteristics. 

 MRFs have been used crucially for modeling \textit{spatially correlated data sets}, which appear naturally in population genetics~\cite{banf2017enhancing, wei2007markov}, image segmentation~\cite{celeux2003procedures,chatzis2008fuzzy}, climate data modeling~\cite{guillot2015climate,prates2022temporal}. In this research in conditional independence and structure testing has gained popularity recently \cite{neykov2019property,schaeben2014mathematical}, it is worth noting that the literature on testing for the independence of two MRFs typically assumes multiple \textit{independent} observations from the MRF distribution, with the number of observations going to infinity; in this paper we are concerned with the depenedent data setting where there is only one draw from each of two MRF distributions with the dimension going to infinity.



In this paper, we derive the first, to our knowledge, characterization of nonsense associations for a class of network-dependent data. 
We consider the asymptotic distributions of estimates and tests of association between two random vectors, $\mathbf{X}$ and $\mathbf{Y}$, where the coordinates of $\mathbf{X}$ and of $\mathbf{Y}$ exhibit dependence but $\mathbf{X}$ is independent of $\mathbf{Y}$. We focus on dependence that is governed by graph or network structure and generated from a Markov random field (MRF) model with pairwise positive dependence. We will demonstrate that under some MRF models, the asymptotic variance of estimates of correlation or covariance is inflated relative to the i.i.d. setting; this is analogous to Yule's nonsense correlations and would result in inflated type I error were i.i.d. tests to be used without correction. However, in other settings dependence results in \textit{deflated} asymptotic variance; this is in stark contrast to Yule's finding and contradicts the heuristic that positive dependence results in smaller effective sample sizes and therefore higher variances. 

In a particularly surprising result, we challenge the common wisdom that positive dependence in the errors of a simple linear regression model should result in standard i.i.d. inference for the regression coefficient being anticonservative. In fact, we show that in order for i.i.d. inferential procedures to be invalid, the variance-covariance matrices of the dependent and the independent variables must share some structure. Otherwise, e.g. when the independent variable is itself i.i.d., i.i.d. inferential procedures achieve nominal type-I error rates. This contradicts the claim, presented in countless regression classes and textbooks, that correlated regression errors should be handled with modified methods in order for inference to be valid \citep{vittinghoff2006regression, gross2003linear,chatterjee2013handbook}. 


\section{Setting}\label{sec:setting}

Consider two independent random vectors observed on $n$ nodes in a network, $\mathbf{X} = (X_1, \ldots , X_n)$ and $\mathbf{Y} = (Y_1, \ldots, Y_n)$, where the nodes exhibit positive pairwise dependence whenever they are connected by a network tie. More precisely, for $\mathcal{Z}\subseteq \mathbb{R}$, we assume $\mathbf{X},\mathbf{Y} \in \mathcal{Z}^n$ are independently distributed with density of the form
\begin{equation}\label{eq:general_density}
    \bP_{\mathbf{Q}_n}(z) \propto \exp{\left(\bz^\top\mathbf{Q}_n \bz\right)}, \quad z \in \mathcal{Z}^{n}
\end{equation}
where $\bQ_n$ is some $n$ by $n$ matrix related to the adjacency matrix of the underlying network, which can be different for $\mathbf{X},\mathbf{Y}$. We drop the subscript $n$ when context allows it.
Such families are typical examples of Markov Random Fields (MRFs) which are a natural class of models for dependent collections of random variables and have been used extensively across an array of disciplines to explore statistical inference in such dependent systems, including to model social networks (see, e.g., \citealp{west2014exploiting, domingos2001mining, ahmed2009recovering, kindermann1980relation,demarzo2003persuasion,daskalakis2019regression,daskalakis2020logistic,kandiros2021statistical,ogburn2020causal}).
The family \eqref{eq:general_density} includes two classical examples which we will study in this paper: Gaussian graphical models corresponding to $\mathcal{X}= \mathbb{R}$ and Ising models corresponding to $\mathcal{X}= \pm 1$. We will extensively focus on these two cases since the distribution of the statistics we aim to study here remains asymptotically standard normal under no dependence within the coordinates of $\mathbf{X}$ or $\mathbf{Y}$.

Throughout we consider a single pair $(\mathbf{X},\mathbf{Y})$ with $n$ going to infinity and inference on the distribution of the $n$ pairs of observations $(X_i,Y_i)$ with dependence across units. This distinguishes our framework from the more typical framework in which a draw from an MRF represents one single, high-dimensional observation and inference would typically require multiple such observations. We  consider three quantities of interest. The sample covariance $T_n$ is defined by 
\begin{equation}\label{eq:define_tn}
    T_n := \frac{1}{n}\sum_{i=1}^n (X_i -\overline{\bX}_n) (Y_i -\overline{\mathbf{Y}}_n),
\end{equation}
where $\overline{\bX}_n:=\frac{1}{n}\sum_{i=1}^{n}X_i$, $\overline{\mathbf{Y}}_n:=\frac{1}{n}\sum_{i=1}^{n}Y_i$ are the sample means. We will also consider the sample correlation 
\begin{equation}\label{eq:define_rho}
    \rho_n := \frac{\sum_{i=1}^n (X_i -\overline{\bX}_n) (Y_i -\overline{\mathbf{Y}}_n)} 
    { \sqrt{\sum_{i=1}^n(X_i -\overline{\bX}_n)^2} \sqrt{\sum_{i=1}^n(Y_i -\overline{\mathbf{Y}}_n)^2}}.
\end{equation}
 Although tests of independence between $\bX$ and $\mathbf{Y}$ can be done using the sample covariance, it is of interest to note when the asymptotic behavior of $\rho_n$ differs from that of $T_n$. Finally, for continuous data we also consider $\hat\beta$, the estimated coefficient for $X$ in a simple linear regression of $Y$ on $X$.

 In all of the settings we consider below, it is easy to show that the asymptotic null distribution of both $T_n$ and $\rho_n$ under an i.i.d. data-generating process is $N(0,1)$. Similarly, when both $\bX$ and $\mathbf{Y}$ are i.i.d., $\hat\beta$ has a known normal asymptotic distribution.  Unlike the time series setting, where the null distribution of $\rho_n$ under dependence generated by Weiner processes is not generally Gaussian \citep{ernst2017yule}, we find that in many settings with MRF-generated dependence, Gaussianity is preserved for various measures of association. Even in these settings, we find that the asymptotic variance of $\rho_n$ may be inflated or deflated relative to the nominal value of $1$. When variance is inflated, the type I error of any procedure that ignores dependence will be inflated, which is another hallmark of nonsense associations.

 Proofs of all new results are in Section \ref{sec:proofs}.

\section{Results for binary data}\label{subsec:ising}

Let $\mathbf{Z}=(z_1,\ldots,z_n)^\top\in \{\pm 1\}^n$ be a random vector with the joint distribution of $\mathbf{Z}$ given by an Ising model defined as:
\begin{align}
\bP_{\beta, \bQ_n}(\mathbf{Z}=\bz):=\frac{1}{Z(\beta,\mathbf{Q}_n)}\exp{\left(\frac{\beta}{2}\bz^\top\mathbf{Q}_n \bz\right)},\qquad \forall \bz \in \{\pm 1\}^n,
\label{eqn:general_ising}
\end{align}
where $\mathbf{Q}_n$ is an $n \times n$ symmetric and hollow matrix, $\beta \ge 0$ is a non-negative real number usually referred to as the ``inverse temperature," and $Z(\beta,\mathbf{Q}_n)$ is a normalizing constant.  
The pair $(\beta,\mathbf{Q}_n)$ characterizes the dependence among the coordinates of $\bX$, and $X_i$'s are independent if and only if $\beta\mathbf{Q}_n=\mathbf{0}_{n\times n}$.  The matrix $\mathbf{Q}_n$ will be associated with a certain labeled graph: Let $\mathbb{G}_n=(V_n,E_n)$ with vertex set $V_n=\{1,\dots,n\}$ and undirected edge set $E_n \subseteq V_n\times V_n$, then the corresponding $\mathbf{Q}_n= n G_n / 2|E_n| $, where $G_n$ is the adjacency matrix of $\mathbb{G}_n$.\par

First, we present results for sparse graphs, specifically for lattices, where we see that (in high temperature regimes) both the sample correlation and sample covariance exhibit variance inflation. This result generalizes nonsense associations from time series, where dependence is informed by distance in $\mathbb{R}^1$, to higher dimensional dependence such as would be found in spatial or network data, and we believe to be the first such generalization beyond one-dimensional dependence.  We suspect these results will hold for other connected, bounded-neighborhood graphs, but we leave the proof to future work.

We next present results for complete and dense (i.e., $E_n=\Theta(n^2)$) regular graphs. This  includes some specific examples of mean-field type models, see e.g. \citep{basak2015universality,jain2018mean}. In a surprising contrast to other settings, dependence in these dense graphs sometimes results in variance \textit{deflation} for the sample covariance, while the sample correlation has the same asymptotic distribution as under independence. 

\subsection{Lattice models}
Let the data points $i=[n]$ be vertices of $d$-dimensional hypercubic lattice $\Lambda_{n,d}$ and the underlying graph $\mathbb{G}_n$  be the nearest neighbor (in sense of Euclidean distance) graph on these vertices. 
This nearest neighbor graph corresponds to the matrix $\bQ_{ij}=\mathbf{1}(\|i-j\|_1= 1)$.

The lattice Ising model typically undergoes a ``thermodynamic" phase transition at a value of the temperature parameter that is determined by the dimension $d$ of the lattice: there exists $\beta_c:= \beta_c(d)$ which governs the macroscopic behavior of the system of observations $X_i,i\in \Lambda_{n,d}$ as $n\rightarrow \infty$. In particular, the sample average $\overline{\bX}_n =n^{-1}\sum X_i$ converges to $0$ in probability for $\beta<\beta_c(d)$ and to an equal mixture of two delta-dirac random variables $m_+(\beta)$ and $m_-(\beta)=-m_+(\beta)$, for $\beta>\beta_c(d)$ (\cite{lebowitz1977coexistence}). The value of $\beta_c(d)$ (which is known to be strictly positive for any fixed $d\geq 1$) is hard to derive precisely. For $d=1$, it is known from the first work in this area \citep{ising1925beitrag} that $\beta_c(1)=+\infty$ and consequently the Ising model in 1-dimension is said to have no phase transitions. The seminal work of \cite{onsager1944crystal} derived  $\beta_c(2)=2\log(1+\sqrt{2})$. Obtaining an analytical formula for $\beta_c(d)$ for $d\geq 3$ remains open. We refer the interested reader to the excellent expositions in \cite{friedli2017statistical,duminil2017lectures} for more details on this topic.\par
Now, we state the main result of this section.
\begin{thm}\label{thm:lattice_inflation}
    Suppose $\mathbf{X}\sim \bP_{\beta_1, \mathbf{Q}(\Lambda_{n,d})}$, $\mathbf{Y}\sim \bP_{\beta_2, \mathbf{Q}(\Lambda_{n,d})}$ for some $\beta_1,\beta_2<\beta_c(d)$ with $\mathbf{X}$  independent of $\mathbf{Y}$. Then, there exists $v= v(\beta_1,\beta_2) \in \mathbb{R}^+$ such that $\sqrt{n} T_n \xrightarrow{d} N(0,v^2)$ and the asymptotic variance satisfies $v^2 >1$. The sample correlation has the same asymptotic distribution as the sample covariance: $\sqrt{n} \rho_n \xrightarrow{d} N(0,v^2)$. Moreover, $v(\beta_1,\beta_2)$ is a non-decreasing function in either coordinate. 
    
\end{thm}

Obtaining the precise limiting constant $v^2$ in Theorem \ref{thm:lattice_inflation} is an important question and we leave it as future research direction. Since $v^2>1$, our result demonstrates that both covariance and correlation exhibit variance inflation. Moreover, in Section \ref{sec:simulations} we present  evidence from simulations suggesting that $v^2$ is strictly increasing in $\beta_1,\beta_2$ for $0<\beta_1,\beta_2<\beta_c$. This result indicates that a researcher ignoring the dependence in $\bX$ and $\mathbf{Y}$ would perform anticonservative inference, with \textit{type I error increasing as a function of dependence}; this is a direct analogue to nonsense correlations in time series.

\subsection{Complete and dense graphs}
The Ising model on 
the complete graph, where $E_n=V_n\times V_n$, is known as the Curie-Weiss model~\citep{kac1969mathematical,nishimori2001statistical}. This model corresponds to having $\bQ_n:=\bQ^{\rm CW}= \frac{1}{n}(\mathbf{1}\mathbf{1}^\top-I)$. Similar to the lattice Ising model, the Curie-Weiss model also undergoes a phase transition at $\beta_c= 1$. Namely, $\overline{\bX}_n\xrightarrow[]{p}0$ for $\beta \le 1$ and $\overline{\bX}_n\xrightarrow{d} \frac{1}{2}(\delta_m +\delta_{-m})$ if $\beta>1$, i.e. an equal mixture of two Dirac delta random variables with mass on $\pm m$ where $m=m(\beta)$ is the unique positive root of the equation $m=\tanh(\beta m)$. 
Under this model, we precisely characterize the asymptotic distribution of sample covariance $T_n$ and sample correlation $\rho_n$.

\begin{thm}\label{thm:cw_tn}
 Suppose $\bX\sim \bP_{\beta_1, \bQ^{\text{CW}}}$, $\mathbf{Y} \sim \bP_{\beta_2, \bQ^{\text{CW}}}$ and $\bX$ independent of $\mathbf{Y}$. Then 
 \begin{equation}\label{eq:cw_clt}
     \sqrt{n} T_n \xrightarrow{d} N(0,(1-m^2_1)(1-m^2_2)),
 \end{equation}
 where if $\beta_i \le 1$, then $m_i=0$ and otherwise $m_i$ is the unique positive root of the equation $m=\tanh(\beta_i m)$, $i=1,2$. 
 Noting that $m_i$ must lie in $[-1,1]$, this implies $\lim\limits_{n \rightarrow \infty}\text{Var}(\sqrt{n}T_n)=1$ iff $\beta_1,\beta_2 \le 1$, otherwise $\lim\limits_{n \rightarrow \infty} \text{Var}(\sqrt{n}T_n) < 1$.
\end{thm}

The implication of the above theorem is two-fold: first, it shows that the distibution of (scaled) sample covariance is asymptotically normal irrespective of the choice of $\beta_1,\beta_2$. Second, it characterizes asymptotic variance as a function of $\beta$ and shows a phase transition at $\beta_i = 1$, $i=1,2$. In fact, there is a \textit{variance deflation} unless $\max\{\beta_1,\beta_2\} \le 1$. To illustrate our findings, we plot the asymptotic variance in Figure 1 when $\beta_1=\beta_2$.

\begin{figure}[h]\label{fig:var(T)}
	\begin{center}
	\includegraphics[scale=0.3]{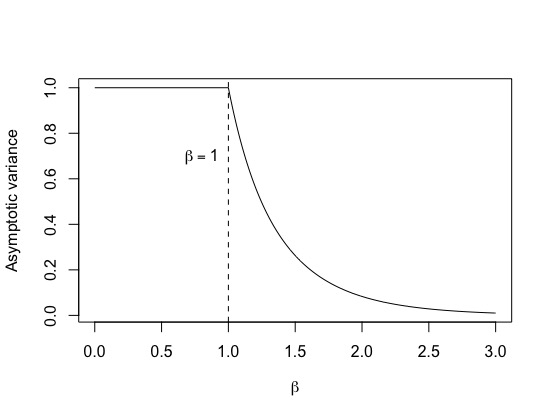}
\caption{Asymptotic variance of $\sqrt{n}T_n$ when $\beta_1=\beta_2=\beta$.}
	\end{center}
\end{figure}

The proof of Theorem \ref{thm:cw_tn} shows how the result can be generalized to a large family of graphs as long as one of the variables is sampled from a Curie-Weiss model. More precisely, if $\bX \sim \bP_{\beta_1,\bQ}$ satisfies $\overline{\bX}^2_n \xrightarrow{\bP} m^2_1$ and $\mathbf{Y} \sim \bP_{\beta_2, \bQ^{\text{CW}}}$ independent of $\bX$, then $\frac{1}{n} \sum_{i=1}^{n}(X_i -\overline{\bX}_n)^2 \xrightarrow{\bP} 1-m^2_1$ and the conclusion of Theorem \ref{thm:cw_tn} still holds. Example of such interaction matrices $\bQ$ are scaled adjacency matrices of dense mean field graphs (cf. \cite[Theorem 1.1]{deb2023fluctuations}) which includes Erd\H{o}s-R\'{e}nyi graphs, balanced stochastic block models, sparse regular graphons, block spin Ising models and sparse graphs like lattice Ising models.

A consequence of Theorem \ref{thm:cw_tn} is that the asymptotic distribution of sample correlation $\rho_n$ is independent of the choice of $\beta_1, \beta_2$. The result is stated via the following corollary:
\begin{corollary}\label{cor:cw_correlation}
Suppose $\bX\sim \bP_{\beta_1, \bQ^{\text{CW}}}$, $\mathbf{Y} \sim \bP_{\beta_2, \bQ^{\text{CW}}}$ and $\bX,\mathbf{Y}$ are independent. Then 
 \begin{equation}
     \sqrt{n} \rho_n \xrightarrow{d} N(0,1),
 \end{equation}
 for any $\beta_1,\beta_2 \ge 0$.
\end{corollary}

The asymptotic variance of sample correlation remains equal to $1$ for any values of $\beta_1,\beta_2$. It is notable that in this setting, unlike the lattice setting, \textit{the asymptotic variances of the sample covariance and correlation differ in nature}. A practitioner unaware of dependence in their data could still perform valid inference about the association between $X$ and $Y$ using the sample correlation, but not necessarily using the sample covariance.

These results for the Curie-Weiss model can be extended to mean-field type models (cf. \citep{basak2015universality,jain2018mean}) for dense regular graphs. We make the following assumptions on the dependency matrix $\bQ_n$:
\begin{ass}\label{ass:reg}
    $\bQ_n$ is a regular matrix, i.e., for all $i=1,2, \ldots, n$, we have $\sum_{j=1}^{n} Q_{n,ij}=1$.
\end{ass}
\begin{ass}\label{ass:entry_bound}
    There exists a constant $\kappa>0$ such that $$\max_{1 \le i,j \le n} Q_{n,ij} < \frac{\kappa}{n}.$$
\end{ass}
\begin{ass}\label{ass:frob_conv}
    There exists a constant $\gamma \in \mathbb{R}$ such that
    $$\|\bQ_n\|^2_{F}=\sum_{i,j=1}^{n} Q_{n,ij} \rightarrow \gamma.$$
\end{ass}
We will need to assume the sequence of matrices $\{\bQ_n\}$ converge in cut-metric. To this end, we have to introduce the basics of graphons and cut-distance convergence and refer the curious readers to \cite{borgs2008convergent,borgs2012convergent,lovasz2012large}.\par 
A \textit{graphon} is a symmetric bounded measurable function from $[0,1]^2$ to $[0,1]$. Let $\mathcal{W}$ be the set of all graphons. One can equip the space of graphons $\mathcal{W}$ by the cut-distance, given by
\begin{equation*}
    d_{\square}(W_1,W_2):= \sup_{S,T \subseteq [0,1]}\Big|\int_{S \times T} (W_1(x,y)-W_2(x,y))dxdy\Big|,
\end{equation*}
where the supremum is taken over all measurable subsets $S,T$ of $[0,1]$. Define the weak cut distance by $$\delta_\square(W_1,W_2):= \inf_{\sigma} (W^\sigma_1,W_2)= \inf_{\sigma} (W_1,W^\sigma_2)$$ where $\sigma$ ranges from all measure preserving bijections $[0,1]\rightarrow [0,1]$ and $W^\sigma(x,y)=W(\sigma(x),\sigma(y))$. \par
Given an $n \times n$ matrix $A_n$, we can define the corresponding graphon $W^{A_n}$ as a piecewise constant function given by
$$W^{A_n}(x,y)= A_{n, \lceil nx \rceil, \lceil ny \rceil}.$$
We call a set of matrices $A_n$, $n \in \bN$ \textit{converge in cut distance} to a graphon $W \in \mathcal{W}$ if $\delta_{\square}(W^{A_n},W) \rightarrow 0$.\par
Finally, given $W \in \mathcal{W}$, define its Hilbert-Schmidt operator $T_W$ from $[0,1]$ to $[0,1]$ given by 
\begin{equation*}
    T_W(f)(x)= \int_{[0,1]} W(x,y)f(y) dy, \quad x\in [0,1].
\end{equation*}
Since the operator $T_W$ is a compact operator, it is has countably many eigenvalues, which we denote by $\lambda_1, \lambda_2, \ldots$ sorted by decreasing order in absolute values $|\lambda_1| \ge |\lambda_2| \ge \ldots$.\par
Now, we are in a position to state our final assumption on the sequence of matrices $\{\bQ_n\}$:
\begin{ass}\label{ass:graphon_limit}
    The sequence of matrices $\{n\bQ_n\}$ converges in cut distance to a graphon $W$ (i.e., $\delta_{\square}(W^{n \bQ_n},W) \rightarrow 0$). Moreover, if $T_W$ has eigenvalues $\lambda_1, \lambda_2, \ldots$, we have $\lambda_1 > \sup_{i \ge 2} \lambda_i$.
\end{ass}
Let us briefly discuss the spectral gap condition we have assumed here. By our regularity Assumption \ref{ass:reg}, the largest eigenvalue of $\bQ_n$ is exactly $1$. This, along with \cite[Theorem 11.54]{lovasz2012large} yields that $\lambda_1=1$. Hence, we are assuming $\sup_{i \ge 2}\lambda_i <1$. Since $|\lambda_i| \le 1$, we allow the possibility that $\lambda_2=-1$, which can happen for regular connected bipartite graphs. One can expect any universality result regarding the asymptotic distribution of sample covariance might fail if one does not assume such a spectral gap condition. \par 
Before we mention our main result of this section, we briefly provide some examples of matrices $\bQ_n$ which satisfy our Assumptions \ref{ass:reg}-\ref{ass:graphon_limit}. The examples arise as the scaled adjacency matrices of regular graphs. First, note that the previously discussed Curie-Weiss model, $\bQ_n = \frac{1}{n} (\mathbf{1}\mathbf{1}^\top-I)$ satisfies all the assumptions. This is because the corresponding graphons $f^{\bQ_n}$ converges in cut distance to the constant graphon $W \equiv 1$, yielding $\lambda_1=1$ and $\lambda_i=0$ for $i \ge 2$. Second, the scaled adjacency matrix of the complete bipartite graph also satisfies the Assumptions \ref{ass:reg}-\ref{ass:graphon_limit}. To see this, note that,
\begin{equation*}
    Q_{n,ij}= \begin{cases}
         \frac{2}{n} \text{    if } 1\leq i \leq \frac{n}{2} \text{ and } \frac{n}{2}+1 \leq j \leq n \\
        \frac{2}{n} \text{    if } 1\leq j \leq \frac{n}{2} \text{ and } \frac{n}{2}+1 \leq j \leq n \\
         0 \text{ otherwise.}
    \end{cases}
\end{equation*}
This is the (scaled) adjacency matrix of a regular graph with degree $\frac{n}{2}$. Also, $n\bQ_n$ converges, in cut norm to the limiting graphon $W$ given by
\begin{equation*}
    W(x,y) = \begin{cases}
        2 \text{ if } 0<x<\frac{1}{2} \text{ and } \frac{1}{2} <y <1, \\
        2 \text{ if } \frac{1}{2}<x<1 \text{ and } 0<y< \frac{1}{2} ,\\
        0 \text{ otherwise}.
    \end{cases}
\end{equation*}
This limiting graphon $W$ has two non-zero eigenvalues $\lambda_1= 1, \lambda_2= -1$ satisying Assumption \ref{ass:graphon_limit}. In fact, our assumptions continue to hold true for general complete multipartite graphs. For our final example, we consider random $d_n$ regular graphs, where $\frac{d_n}{n} \rightarrow d \in (0,1)$ and $Q_{n,ij}= \frac{1}{d_n} \mathbf{1}_{(i,j) \in E_n}$. By definition, $\bQ_n$ is regular, with $n\bQ_n$ converging in cut-norm to the constant graphon $W \equiv 1$ \cite[Theorem 1.1]{chatterjee2011random}, hence the only non-zero eigenvalue is $\lambda_1=1$.

Note that the Ising model with interaction matrix $\bQ_n$ satisfying Assumptions \ref{ass:reg}-\ref{ass:graphon_limit} satisfies the exact same phase transition as the Curie-Weiss model $\bP_{\beta,\bQ^{\rm CW}}$ at $\beta_c=1$. We now state our universality result. 

\begin{thm}\label{prop:er_var}
    Let $\{\bQ_n\}$ be a sequence of matrices satisfying Assumptions \ref{ass:reg}-\ref{ass:graphon_limit}. Suppose $\bX\sim \bP_{\beta_1, \bQ_n}$ and $\mathbf{Y} \sim \bP_{\beta_2, \bQ_n}$, independent of $\bX$. Then   $\sqrt{n} T_n \xrightarrow{d} N(0,(1-m^2_1)(1-m^2_2))$,
 where $m_i=0$ if $\beta_i \le 1$, and otherwise $m_i$ is the unique positive root of the equation $m=\tanh(\beta_i m)$, $i=1,2$.
 Moreover, the sample correlation has asymptotic distribution $\sqrt{n} \rho_n \xrightarrow{d} N(0,1)$.
\end{thm}
The above result shows that the asymptotic distribution of (scaled) sample covariance is the same for all dense regular graphs (under Assumptions 1-4). Therefore, the phenomenon we observed in the previous section, i.e., variance deflation for sample covariance, is not limited to complete graphs. It is very natural to expect such a result to continue to hold for other mean-field graphs like dense Erd\H{o}s-R\'{e}nyi graphs. However, the current proof crucially relies on the fact that $\bQ_n$ is \textit{exactly} regular and we leave any such extension for future research. \par

These results illustrate an important and surprising dichotomy between dense regular graphs and lattice Ising models. While on lattices we see the familiar phenomenon of variance inflation for both sample correlation and sample covariance, analogous to nonsense correlations in time series, the dense regime  exhibits a variance \textit{deflation} for the sample covariance and no change in distribution for sample correlation.



\section{Results for continuous data}\label{sec:continuous}

We now turn our attention to Gaussian graphical models $\bX, \mathbf{Y} \sim N(0,\Sigma)$, corresponding to $\mathcal{X} =\mathbb{R}$ and $\bQ_n= -\frac{1}{2} \Sigma^{-1}$ in model \eqref{eq:general_density}. 

In addition to studying the asymptotic behavior of $T_n$ and $\rho_n$ under this model, we also study the behavior of the coefficient $\beta$ for $X$ in a simple linear regression of $Y$ on $X$, allowing $\mathbf{X}$ and $\mathbf{Y}$ to have different covariance matrices: $\mathbf{Y} \sim N(0,\Sigma_Y)$ and $\mathbf{X} \sim N(0,\Sigma_X)$.
We find that, surprisingly and contrary to dozens of presentations on OLS, inference about $\beta$ need not be anti-conservative when the regression errors exhibit positive dependence.  If $\Sigma_X$ is diagonal, or if a particular condition holds on the product of the eigenvalues of $\Sigma_X$ and $\Sigma_Y$, then inference under OLS is exact, i.e. it achieves the nominal type I error rate.

We find conditions under which $T_n$ and $\rho_n$ are asymptotically normal, and, in line with the phenomenon of nonsense correlations, we find that $\rho_n$ exhibits variance inflation. However, surprisingly, under these same conditions $T_n$ may exhibit either variance inflation or deflation.

\subsection{Correlation and Covariance}

Let $\bX, \mathbf{Y} \sim N(0,\Sigma)$ and define $J:= I_n - \frac{1}{n} \mathbf{1}\mathbf{1}^\top$, where $\mathbf{1}$ is the $n$ length vector of all $1$'s. The asymptotic distribution of $T_n$ and $\rho_n$ will depend on the eigenvalues of $\widetilde{\Sigma}:= \Sigma^{1/2} J \Sigma^{1/2}$. Let $\tilde{\lambda}_1 \ge \tilde{\lambda}_2 \ldots \ge \tilde{\lambda}_n \ge0$ be the eigenvalues of $\widetilde{\Sigma}$ 
in non-increasing order.
The asymptotic distributions of $T_n$ and $\rho_n$ depend on whether the eigenvalues are all of similar order. We formalize this notion in the following Theorem.

\begin{thm}\label{lem:ggm}
Suppose $\bX, \mathbf{Y} \sim N(0,\Sigma)$ are independent. 
Define $a_n:= \frac{\sum_{i=1}^n\tilde{\lambda}_i}{\sqrt{ n\sum_{i=1}^n\tilde{\lambda}^2_i}}$.
\begin{enumerate}
    \item[(i)] If $\tilde{\lambda}^2_1 \ll \sum_{i=1}^{n}\tilde{\lambda}^2_i$, then
    \begin{equation*}\label{eq:small_eval}
        \frac{nT_n}{\sqrt{\sum_{i=1}^n\tilde{\lambda}^2_i}}\xrightarrow{d} N(0,1)
    \end{equation*}
and
\begin{equation*}\label{eq:ggm_rho}
       \sqrt{n} \rho_n a_n \xrightarrow{d} N(0,1);
    \end{equation*}
hence, $\liminf_{n \rightarrow \infty} \text{Var} (\sqrt{n} \rho_n) \ge 1$.
    \item[(ii)] On the other hand, if $\tilde{\lambda}_1 \gg n \tilde{\lambda}_2$, then 
    \begin{equation*}\label{eq:big_eval}
        \frac{nT_n} {\tilde\lambda_1} \xrightarrow{d} AB, 
    \end{equation*}
    where $A \sim N(0,1)$, $B \sim \sqrt{\chi^2_1}$ and $A,B$ are independent. Moreover, $\rho_n$ converges weakly to a Rademacher distribution, i.e.,
    \begin{equation*}
        \rho_n \xrightarrow{d} \frac{1}{2} (\delta_{1}+\delta_{-1}).
    \end{equation*}
\end{enumerate}
\end{thm}

Since $a_n \le 1$, the Theorem implies variance inflation for the sample correlation. In fact, for any $ \sigma^2\ge 1$, one can construct a sequence of covariance matrix $\{\Sigma_n\}$ such that $\sqrt{n} \rho_n \xrightarrow{d} N(0, \sigma^2)$. To see this, set $l_0= \lceil \frac{n}{\sigma^2} \rceil$ and set $\tilde \lambda_1 = \ldots= \tilde \lambda_l=1$ and $\tilde\lambda_i=0$ if $i > l_0$. Then $a_n \rightarrow \frac{1}{\sigma}$. Hence, if $\frac{1}{\sqrt{n}}\mathbf{1}, v_1,\ldots, v_{n-1}$ forms an orthonormal basis, choosing $\Sigma_n:= \frac{1}{n}\mathbf{1}\mathbf{1}^\top+ \sum_{i=2}^n \tilde\lambda_{i-1} v_iv^\top_i$ would yield $\sqrt{n}\rho_n \xrightarrow{d} N(0,\sigma^2)$.

 Although the sample correlation exhibits variance inflation, the  sample covariance can exhibit either variance inflation or deflation. Corollary \ref{lem:contagion} below provides an example of variance deflation; to see example of variance inflation choose $\tilde \lambda_1 = \sqrt{n} \log n$ and $\tilde{\lambda}^2_1 \ll \sum_{i=1}^{n}\tilde{\lambda}^2_i$. This will yield $\lim\limits_{n \rightarrow \infty} \text{Var} (\sqrt{n} T_n) \rightarrow \infty$. In future work it will be of interest to characterize the settings in which $T_n$ and $\rho_n$ show divergent asymptotic behavior, with the latter exhibiting variance inflation while the former exhibits deflation.

Theorem \ref{lem:ggm} part (ii) shows more pathological behavior when one eigenvalue is of larger order than the others. In this case the limiting distribution of $T_n$ is not normal and the variance will typically be larger than $1$, and the sample correlation converges to a bimodal Rademacher distribution centered at, but concentrated away from, the true value of 0. These phenomena are analogous to nonsense associations in other settings, where measures of association have been shown to follow asymptotic distributions that are overdispersed relative to the Normal distribution \cite{ernst2017yule} or to be concentrated away from the truth \cite{lee2021network}.

For the special case of equicorrelation we have the following Corollary to Theorem 4: 
\begin{corollary}\label{lem:contagion}
Suppose $\bX,\mathbf{Y} \sim N(0, \Sigma)$ are independent, where $\Sigma$ is an equicorrelation matrix, i.e., $\Sigma= (1-\rho)I + \rho \mathbf{1} \mathbf{1}^\top$, $0<\rho<1$.
\begin{enumerate} 
    \item[(i)] The sample covariance satisfies $\sqrt{n}T_n \xrightarrow{d} N(0,(1-\rho)^2)$. 
    \item[(ii)] $\sqrt{n}\rho_n \xrightarrow{d} N(0,1)$.
\end{enumerate}
\end{corollary}

The above result exhibits an interesting phenomenon: the asymptotic variance of $\sqrt{n}T_n$ is decreasing in $\rho$ but the asymptotic variance of $\sqrt{n}\rho_n$ is independent of the choice of $\rho$.

\subsection{Ordinary Least squares\protect\footnote{The authors are grateful to Youjin Lee, Brian Gilbert, and Abhirup Datta for their work on simulations and informal results that preceded the formal results below.}
}

Let $\mathbf{X}$ and $\mathbf{Y}$ be multivariate normal with simultaneously diagonalizable variance matrices $\Sigma_Y=V\Lambda_YV^T$ and $\Sigma_X=V\Lambda_XV^T$. Then, letting $\boldsymbol{\varepsilon}$ be the error after subtracting the least squares projection of $\mathbf{Y}$ onto $\bX$, $\boldsymbol{\varepsilon}\perp\bX$ is also multivariate normal and with simultaneously diagonalizable with variance matrix $\Sigma_\varepsilon=V \Lambda_\varepsilon V^T$ and 
\begin{align*}
\mathbf{Y}=\mathbf{X}\beta+\boldsymbol{\varepsilon}.
\end{align*}
 In order to derive interpretable results, we assume $\lambda_{X,i}=f(i/n)$ and $\lambda_{\varepsilon,i}=g(i/n)$ for integrable functions $f,g: [0,1]\rightarrow \mathbb{R}_+$.





 We will assess when inference based on standard i.i.d. inference is nevertheless valid, via an asymptotic comparison of the naive  to the true variance of  $\hat{\beta}_{\rm ols}$, where the estimated (naive) standard error is given by the standard i.i.d. formula  $\widehat{\text{Var}_{\text{naive}}}= \frac{\boldsymbol{\varepsilon}^T(\mathbb{I}-P_{\bX})\boldsymbol{\varepsilon}}{n \|\bX\|^2 } =\frac{\boldsymbol{\varepsilon}^T(\|\bX\|^2- \bX\bX^\top)\boldsymbol{\varepsilon}}{n \|\bX\|^4 }$,

\begin{thm}\label{thm:ols} Assume $\mathbf{X}$ and $\mathbf{Y}$  are simultaneously diagonalizable with eigenvalues as described above. Then
 naive i.i.d. inference for $\beta$ is valid, achieving asymptotic Type I error less than or equal to a nominal level, if and only if the following condition holds:
\begin{align}\label{OLScondition}
 \int f(x)g(x)dx \leq \int f(x)dx\int g(x)dx.
\end{align}
    
\end{thm}

 Notably, this condition is met whenever $\bX$ is i.i.d., even if $\mathbf{Y}$ and $\varepsilon$ are not. In this case ratio of the true to the naive variance is asymptotically equal to $1$, meaning that inference using the naive variance estimator will asymptotically exactly achieve the nominal type I error rate. 

If $f,g$ are both increasing or decreasing, \eqref{OLScondition} is violated (by the FKG inequality), yielding anticonservative inference. However, if one is increasing and the other is decreasing, then the condition is met and niave i.i.d. inference is valid.


Note that even when \eqref{OLScondition} holds and naive inference is valid or exact, this inference is not typically \textit{efficient} (unless $\Sigma_\varepsilon$ is the identity matrix). The efficient estimator in these cases is generalized least squares (GLS). 
Under some conditions the estimated regression coefficient from generalized least squares (GLS) will coincide with that from ordinary least squares (OLS) \citep{puntanen1989equality}; in future work it will be interesting to explore the relationship between the conditions under which OLS inference is valid and the conditions under which OLS estimates are efficient.

\section{Numerical Experiments}\label{sec:simulations}

Figures \ref{fig:correlation} and \ref{fig:covariance} illustrate the claims in Theorem 3. We fixed $G_n$ to be a $1000$ by $1000$ lattice and used the Potts package in R to simulate the Ising model on $G$, varying $\beta$ between $0$ and $\beta_c=1.76$. 

Figure \ref{fig:Thm4} illustrates the Rademacher distribution of Theorem 4 (ii).

Figure \ref{fig:OLS} shows confidence interval of $\beta$ under $X, Y$ multivariate normal with non-diagonal covariance matrix. The columns correspond to $\Sigma_X$ having eigenvalues $(i/n)^2$, $i=1,2,\ldots,n$ and $\Sigma_Y$ have eigenvalues (a) $e^{i/n^{0.85}}$, (b) $e^{i/n}$, (c) $e^{-i/n}$, (d) $e^{-i/n^{0.85}}$, (e) $\Sigma_X=I$. In (a) and (b) we have anticonservative inference, in (c) and (d) we have conservative inference, and in (e) inference is exact (though not efficient). This illustrates the theoretical findings in Theorem \ref{thm:ols}.

\begin{figure}[h]
	\begin{center}
	\includegraphics[width=.75\linewidth]{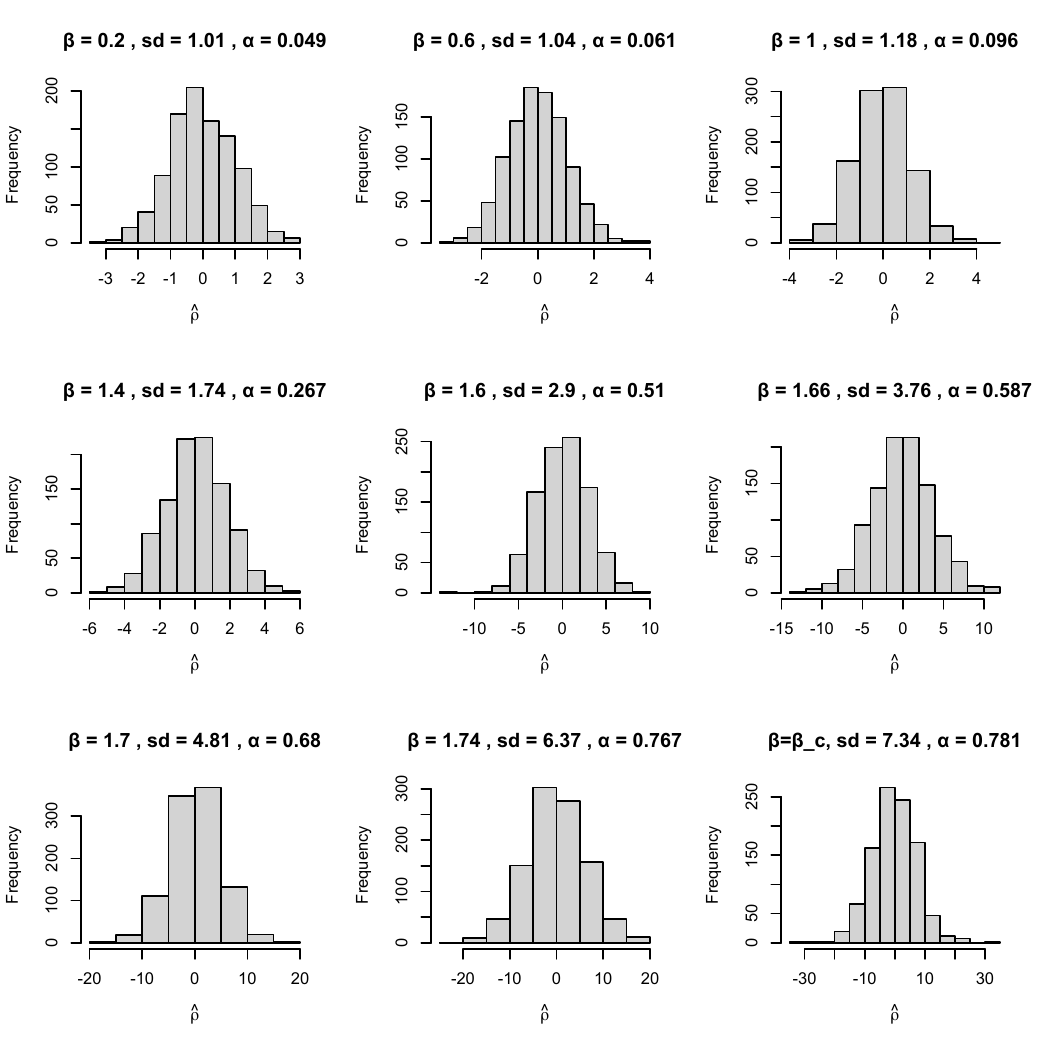}
\caption{Empirical distribution of $\sqrt{n}\rho_n$ over 1000 simulation replicates with $n=10,000$. The inverse temperature parameter $\beta$ ranges from $0$ (i.i.d. data) to the critical value. The standard deviation of the distribution increases monotonically with $\beta$, as does the type I error rate $\alpha$ for a naive null hypothesis test of $\rho=0$.}\label{fig:correlation}
	\end{center}
\end{figure}

\begin{figure}[h]
	\begin{center}
	\includegraphics[width=.75\linewidth]{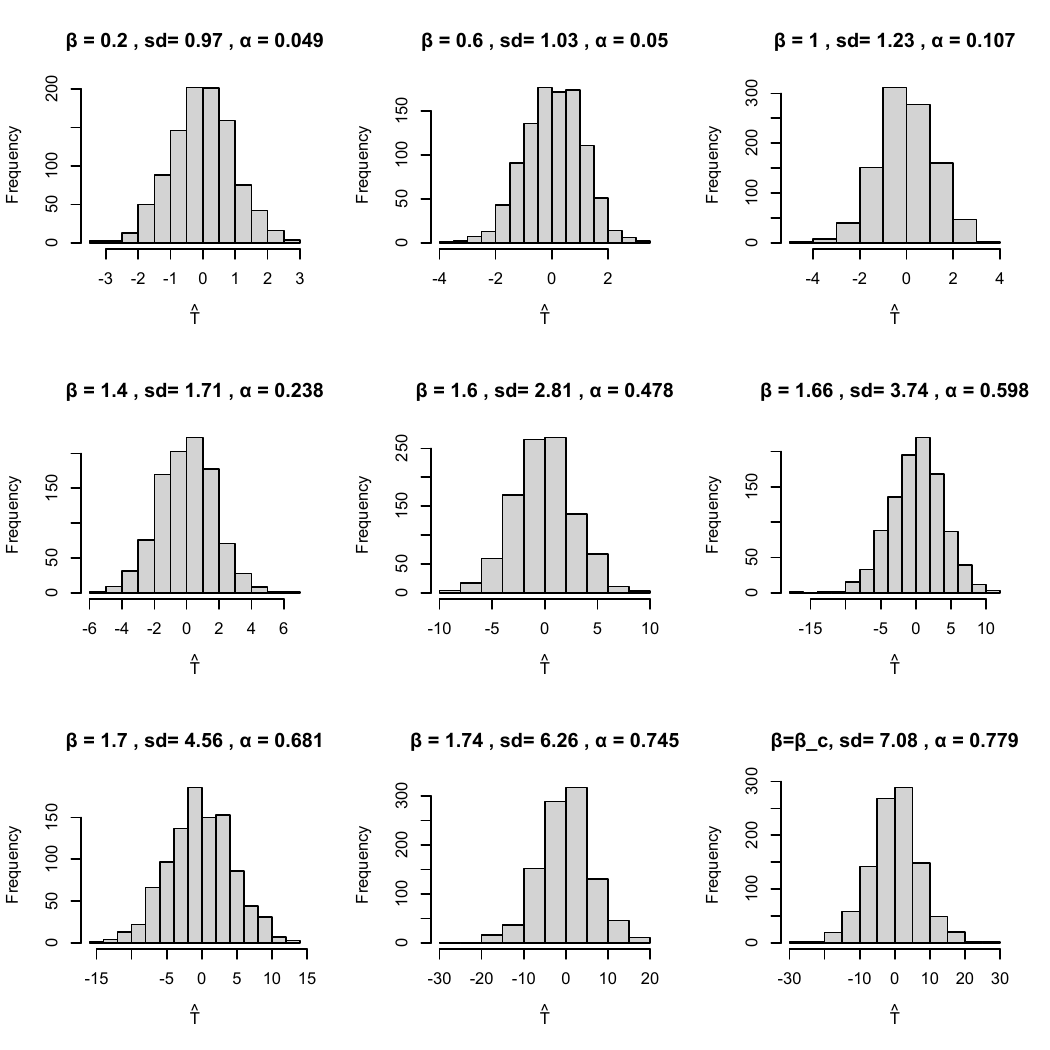}
\caption{Empirical distribution of $\sqrt{n}T_n$ over 1000 simulation replicates with $n=10,000$. The inverse temperature parameter $\beta$ ranges from $0$ (i.i.d. data) to the critical value. The standard deviation of the distribution increases monotonically with $\beta$, as does the type I error rate $\alpha$ for a naive null hypothesis test of $T=0$.}\label{fig:covariance}
	\end{center}
\end{figure}

\begin{figure}[h]
	\begin{center}
	\includegraphics[width=.7\linewidth]{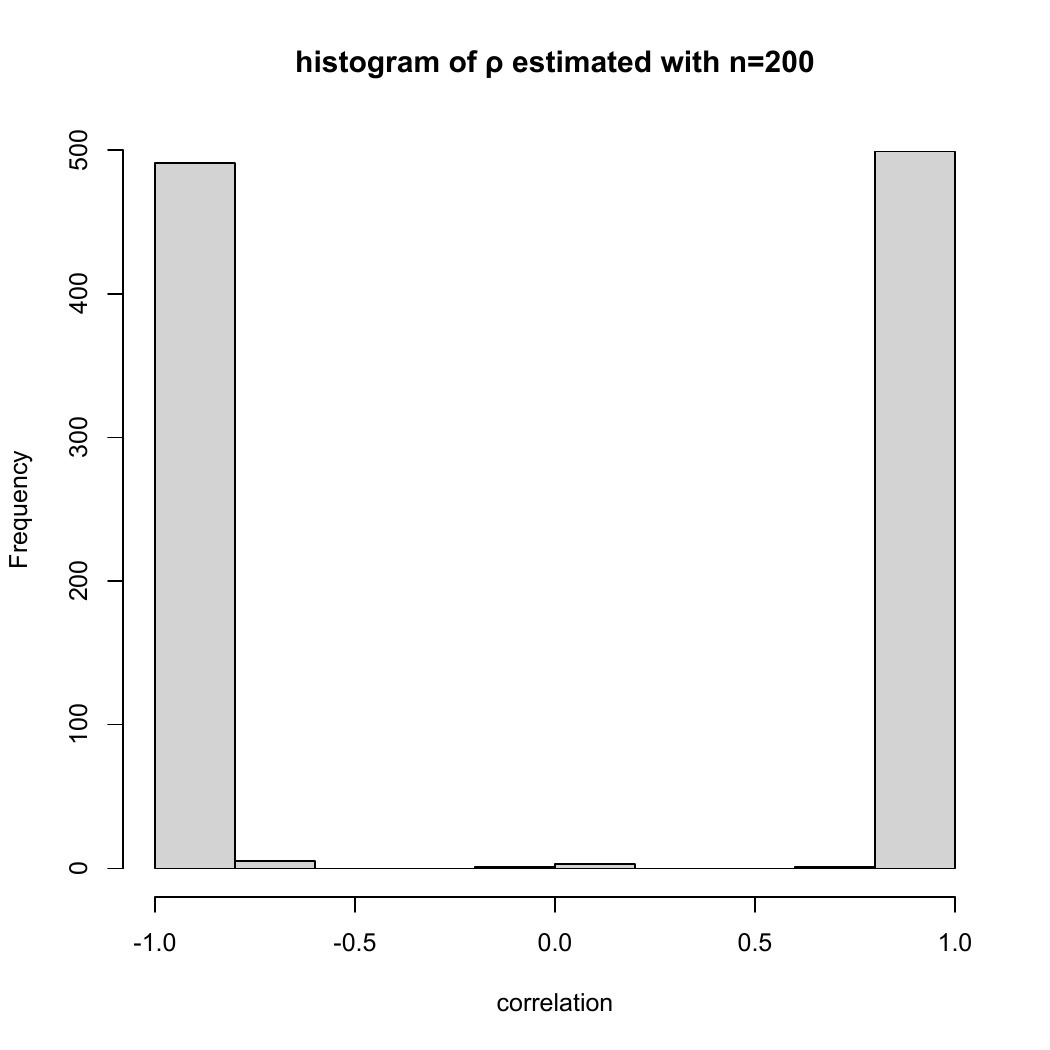}
\caption{Empirical distribution of $\sqrt{n}\rho_n$ over 1000 simulation replicates with $n=200$ with $\mathbf{X},\mathbf{Y} \sim N(0,\Sigma)$ with $\tilde{\lambda_1}\gg n\tilde{\lambda_2.} $}\label{fig:Thm4}
	\end{center}
\end{figure}

\begin{figure}
\centering
\begin{subfigure}{.4\textwidth}
\includegraphics[width=0.9\textwidth]{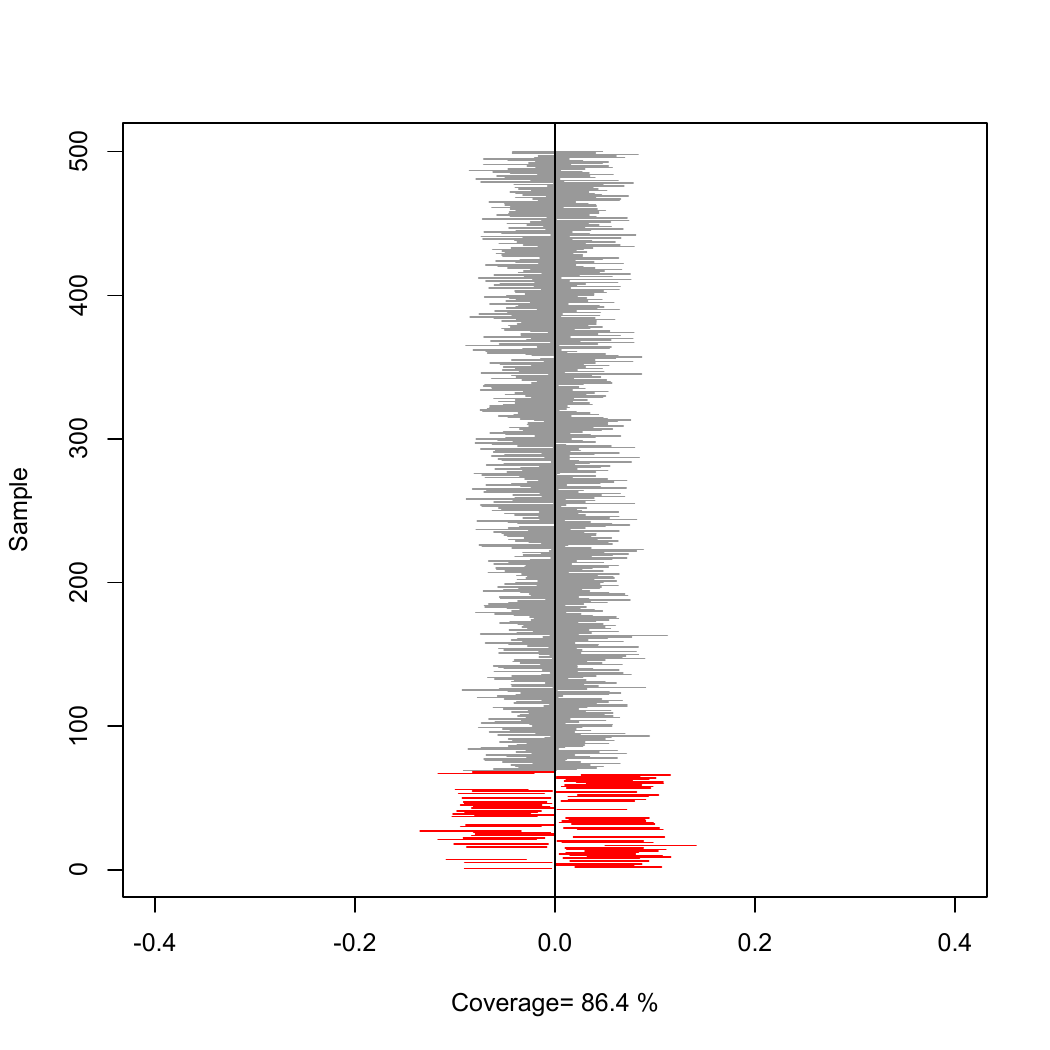}
  \caption{\tiny{Anticonservative inference}}
\end{subfigure}%
\begin{subfigure}{.4\textwidth}
\includegraphics[width=.9\textwidth]{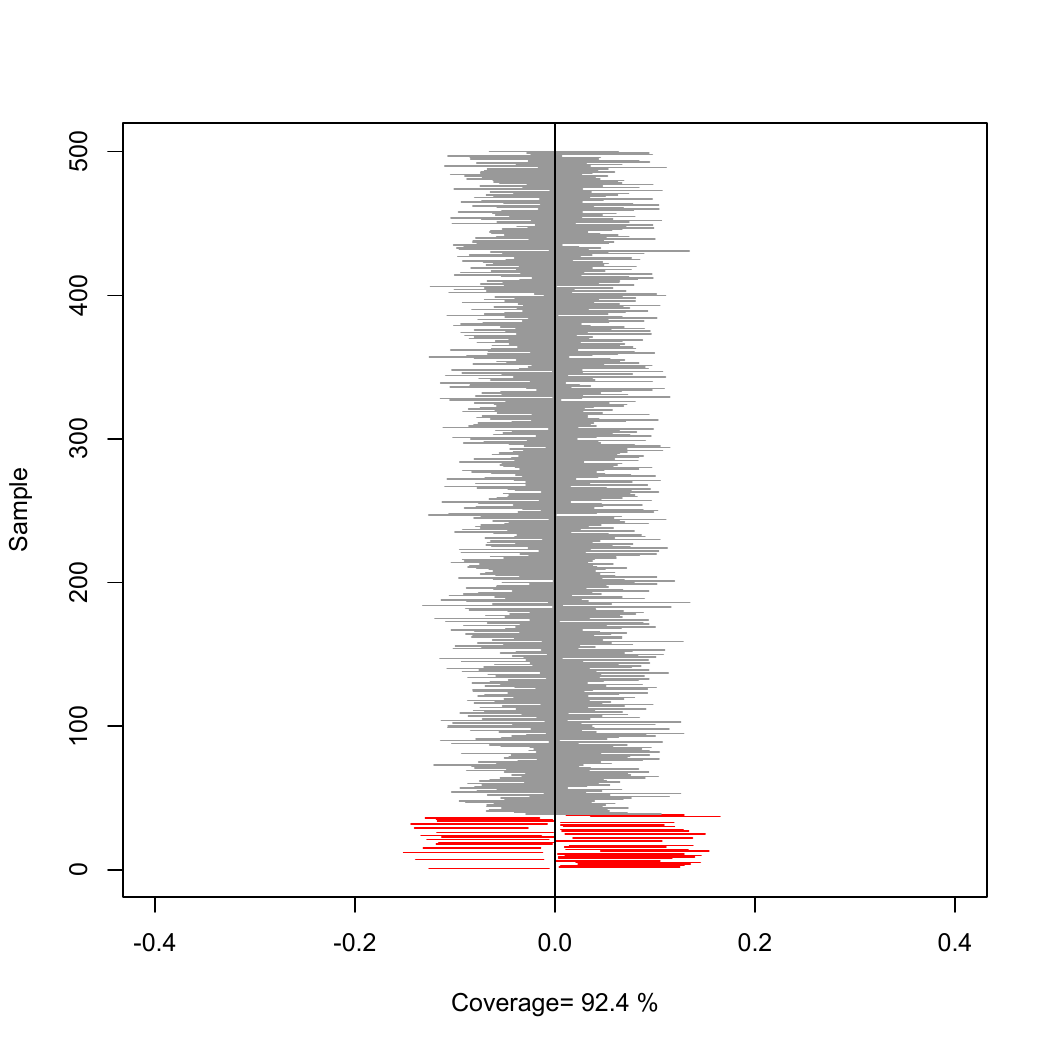}
 \caption{\tiny{Anticonservative inference}}
\end{subfigure}
\begin{subfigure}{.32\textwidth}
\includegraphics[width=.9\textwidth]{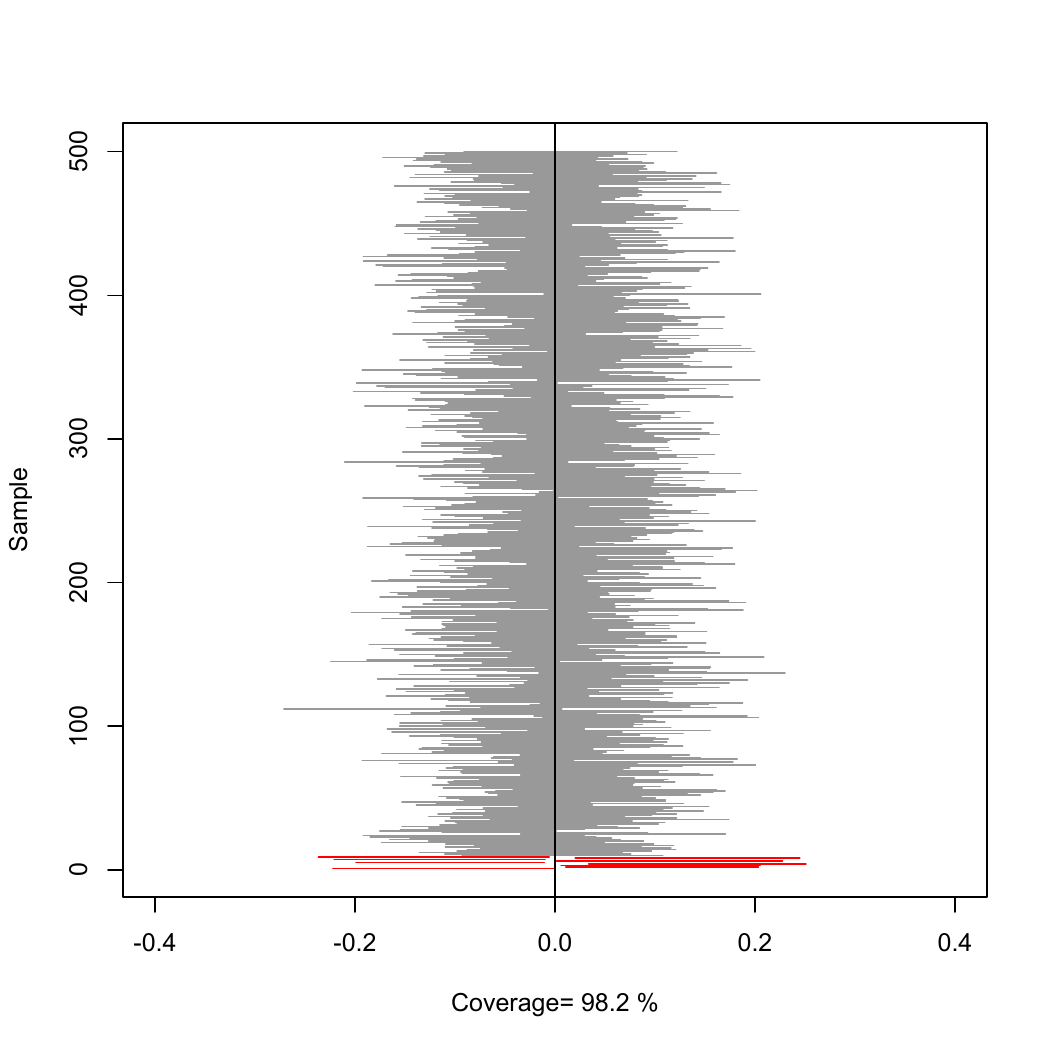}
  \caption{\tiny{Conservative inference}}
\end{subfigure}%
\begin{subfigure}{.32\textwidth}
\includegraphics[width=.9\textwidth]{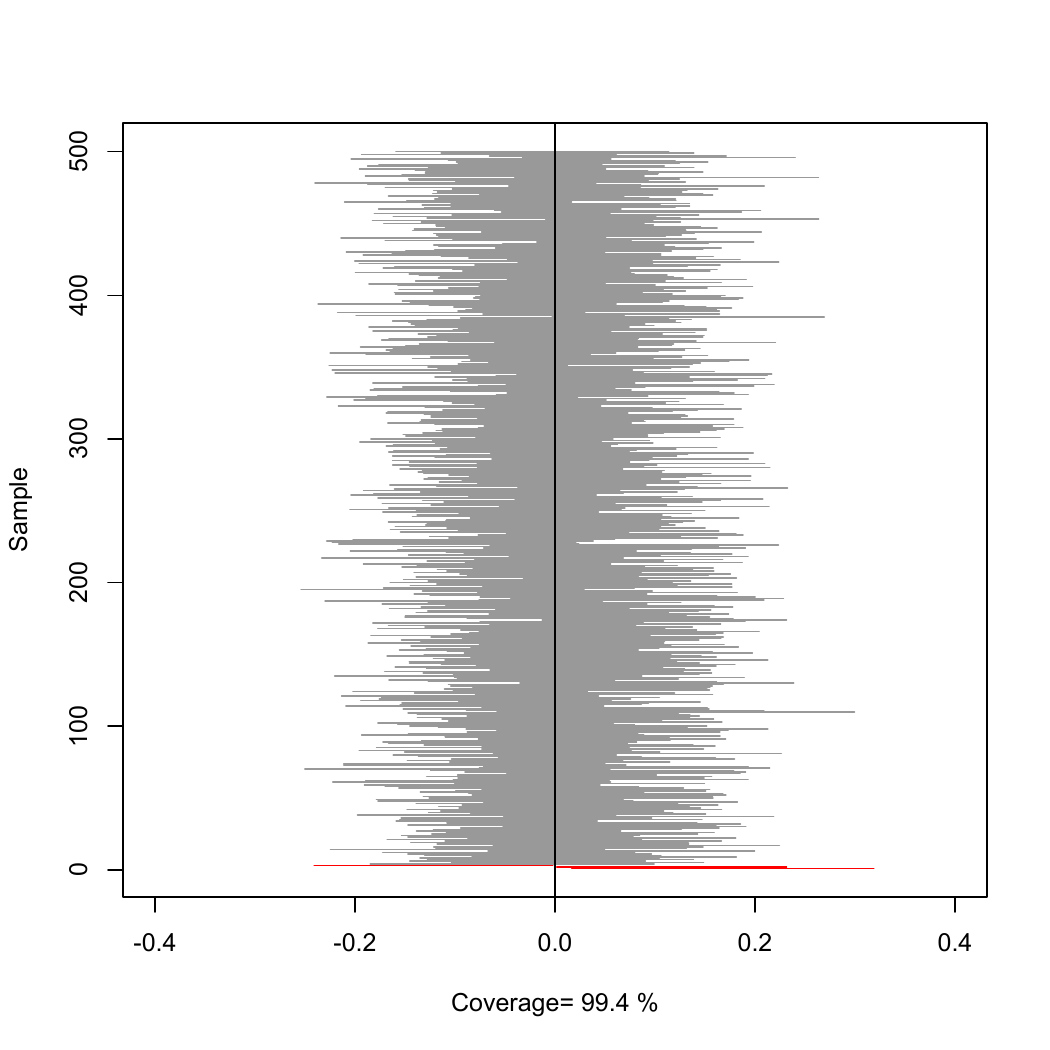}
  \caption{\tiny{Conservative inference}}
\end{subfigure}%
\begin{subfigure}{.32\textwidth}
\includegraphics[width=.9\textwidth]{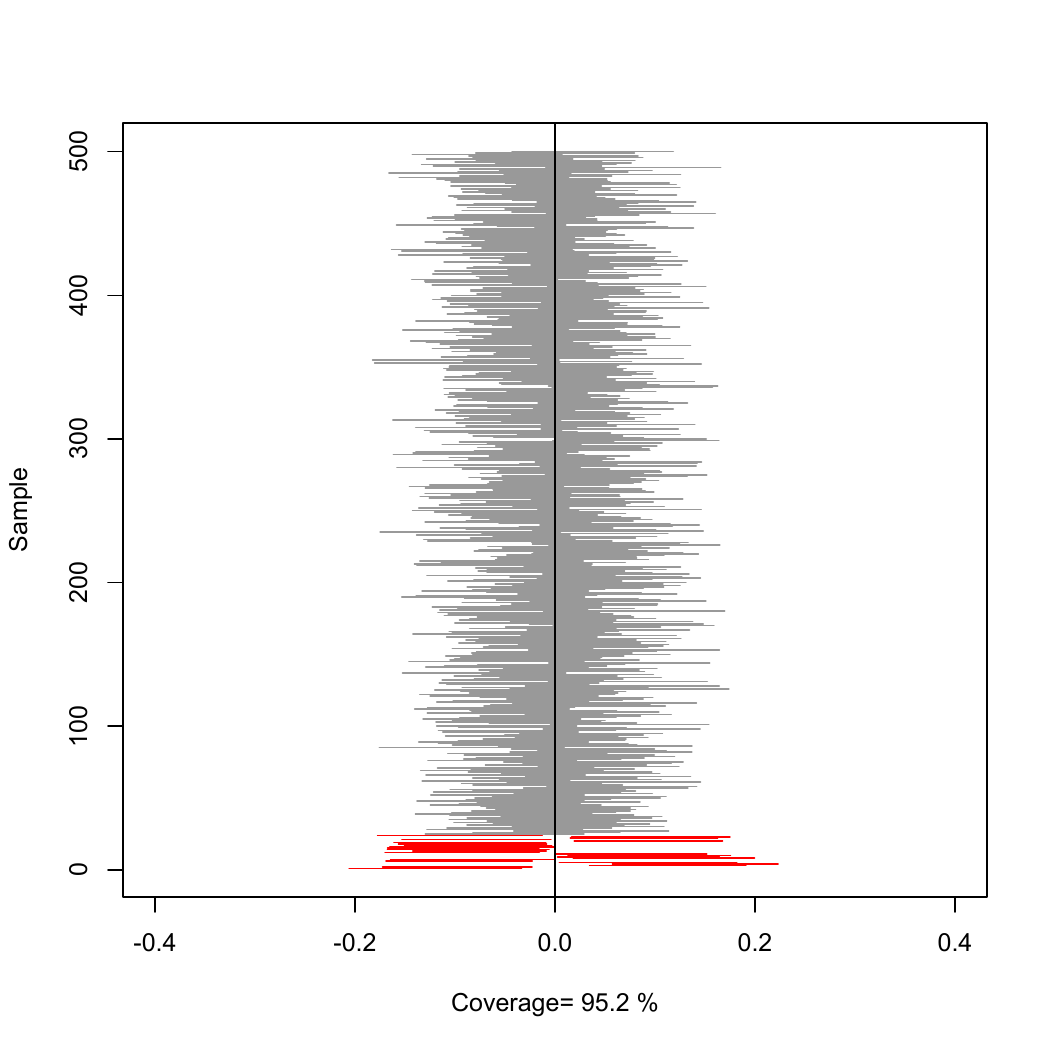}
  \caption{\tiny{Exact inference}}
\end{subfigure}
\caption{Each column corresponds to $95\%$ naive OLS confidence intervals for $\beta$ in the OLS regression of $Y$ onto $X$. In (a) and (b) the eigenvalues of $\Sigma_X$ and $\Sigma_Y$ are both increasing; in (c) and (d) the eigenvalues for $\Sigma_X$ are decreasing and $\Sigma_Y$  increasing. In (e) $\Sigma_X$ is the identity matrix. Each column represents $500$ simulated replications of sample size $n=200$.}
\label{fig:OLS}
\end{figure}

\section{Proofs}\label{sec:proofs}

\subsection{Proof of Theorem \ref{thm:lattice_inflation}}
For notational simplicity, we will assume throughout the proof that $\beta_1=\beta_2=\beta$. The proof for general $\beta_1,\beta_2$ goes through with minor modification.

Recall that $T_n$ is the sample covariance of $\mathbf{X}$ and $\mathbf{Y}$ and $\mathbf{Q}_n$ is the adjaceny matrix on $\Lambda_{n,d}$. To obtain the asymptotic distribution of $\sqrt{n} T_n$, we want to invoke \cite[Theorem 6]{martin1973mixing} which provides a sufficient condition for asymptotic normality. 

Following the notation of \cite[Theorem 6]{martin1973mixing}, we set $A=\{i\}$ with $i \in [n]$ and $X^{(n)}_A= \frac{1}{\sqrt{n}} (X_i- \overline{\bX}_n)$, and $\sigma_A= Y_i$. Then, we have $\sqrt{n} T_n= \sum_{i=1}^n X^{n}_{\{i\}}\sigma_{\{i\}}$. The conditions of the theorem are easily satisfied since, 
\begin{enumerate}
    \item $\sum_{i=1}^{n} (X^{n}_{\{i\}})^2 =\sum_{i=1}^{n} \frac{1}{n} (X_i- \overline{\bX}_n)^2 \le 1$, and
    \item $\max_{i \in n} |X^{n}_{\{i\}}| \le \frac{2}{\sqrt{n}} \rightarrow 0$.
\end{enumerate}
Therefore, taking $h=0$, we have, by \cite[Theorem 6]{martin1973mixing}, $\sqrt{n} T_n \xrightarrow{d} N(0,v)$ once we can show 
\begin{equation}\label{eq:vn_goes_to_v}
    \lim\limits_{n \rightarrow \infty} v_n:=\frac{1}{n}\bE \left(\sum_{i=1}^{n} (X_i-\overline{\bX}_n)Y_i\right)^2 \rightarrow v.
\end{equation}
To show \eqref{eq:vn_goes_to_v}, we have,
\begin{align}\label{eq:lattice_i}
    v_n &= \frac1n \sum_{i=1}^n \bE (X_i -\overline{\bX}_n)^2 +\frac{1}{n}\sum_{i \neq j} \bE (X_i- \overline{\bX}_n)(X_j- \overline{\bX}_n) \bE(Y_i Y_j) \nonumber \\
    &=1- \bE(\overline{\bX}_n^2) + \frac{1}{n}\sum_{i \neq j} \bE (X_i- \overline{\bX}_n)(X_j- \overline{\bX}_n) \bE(Y_i Y_j).
\end{align}
Since $\mathbf{X},\mathbf{Y} \sim \bP_{\beta,\mathbf{Q}}$, we have $\bE(Y_i Y_j)=\bE(X_i X_j)$. Hence, upon expanding, the second summand of \eqref{eq:lattice_i} becomes,
\begin{equation}\label{eq:lattice_term_ii}
    \frac{1}{n}\Big(\sum_{i 
    \neq j} \text{Cov}^2(X_i,X_j) + \bE(\overline{\bX}_n^2) \sum_{i 
    \neq j} \text{Cov}(X_i,X_j) - 2 \sum_{i 
    \neq j} \text{Cov}(X_i,X_j) \bE(X_i \overline{\bX}_n)\Big).
\end{equation}
Now we want to show the third summand of the above display is $o(1)$. To see this, note that,
\begin{align}
    &\frac{1}{n}\sum_{i 
    \neq j} \text{Cov}(X_i,X_j) \bE(X_i \overline{\bX}_n) = \frac{1}{n^2}\sum_{i 
    \neq j} \text{Cov}(X_i,X_j) \times \sum_{l=1}^n \bE(X_iX_l) \nonumber \\
    &=\frac{1}{n^2}\Big(\sum_{i 
    \neq j} \text{Cov}(X_i,X_j) \Big( \sum_{l\neq i,j} \bE (X_iX_l) +1+ \bE (X_i,X_j)\Big) \nonumber \\
    &= \frac{1}{n^2} \Big(\sum_{i \neq j \neq l} \text{Cov}(X_i,X_j) \text{Cov}(X_i,X_l) +  \sum_{i \neq j} \text{Cov}(X_i,X_j) +  \sum_{i \neq j} \text{Cov}^2(X_i,X_j)\Big) \label{eq:ising+ii}
\end{align}

To bound the second and third summands of \eqref{eq:ising+ii}, we begin by stating the following fact: there exists a constant $c = c(\beta,d)$ such that for all $j \in \mathbb Z^d$, 
\begin{equation}
0\le \mathrm{Cov}(X_i,X_j) \le e^{-c \|i-j\|} \label{eq:correlation_decay}
\end{equation}
 if $\beta < \beta_c$ and $\|\cdot\|$ is the $L^1$ norm in $\bZ^d$. The first inequality is known as the GKS inequality (cf \cite{friedli2017statistical}) and the second inequality is proved in \cite{aizenman1987phase,bhattacharya2021sharp,duminil2016new} for $\beta<\beta_c$. Using \eqref{eq:correlation_decay}, one can show that \cite[Theorem 5]{bhattacharya2021sharp} 
 \begin{equation}\label{eq:chi_finite}
     \sum_{i \neq j} \text{Cov}(X_i,X_j) = O(n).
 \end{equation}
 Now, using \eqref{eq:chi_finite}
 and $\text{Cov}(X_i,X_j) \ge 0$ by \eqref{eq:correlation_decay}, we further have $\sum_{i \neq j} \text{Cov}^2(X_i,X_j) = O(n)$, yielding the second and third summand of \eqref{eq:ising+ii} to be $o(1)$. Finally, the first summand of \eqref{eq:ising+ii} can be simplified as
 \begin{align*}
     \frac{1}{n^2}\sum_{i \neq j \neq l} \text{Cov}(X_i,X_j) \text{Cov}(X_i,X_l) &= O\left(\frac{1}{n^2} \sum_{i \neq j} \text{Cov}(X_i,X_j)\right) \\
     &= O\left(\frac{1}{n}\right) = o(1).
 \end{align*}
 This yields the third summand of \eqref{eq:lattice_term_ii} is $o(1)$. Now, we focus on the second summand of \eqref{eq:lattice_term_ii}. Since $\overline{\mathbf{X}} \xrightarrow{\bP} 0$ as $\beta <\beta_c$, by the dominated convergence theorem, we have $\bE(\overline{\bX}_n^2) \rightarrow 0$. Again, invoking \eqref{eq:chi_finite}, we have $\sum_{i 
    \neq j} \text{Cov}(X_i,X_j)= O(n)$, yielding the second summand to be $o(1)$. Plugging this into \eqref{eq:lattice_i}, we have

    \begin{equation}\label{eq:vn_value}
        v_n= 1+\frac{1}{n}\sum_{i 
    \neq j} \text{Cov}^2(X_i,X_j) +o(1)
    \end{equation}
Now we claim the following: $\lim\limits_{n\rightarrow \infty}\frac{1}{n}\sum_{i 
    \neq j} \text{Cov}^2(X_i,X_j)$ exists and is strictly positive. This will show asymptotic normality and the fact that the asymptotic variance is strictly larger than $1$, concluding our proof.

    To see the first part of the claim, note that $\lim\limits_{n\rightarrow \infty}\frac{1}{n}\sum_{i 
    \neq j} \text{Cov}(X_i,X_j)$ exists using \cite[Theorem 5]{bhattacharya2021sharp}. Since $0\le \mathrm{Cov}^2(X_0,X_j) \le e^{-2c \|j\|}$, we have the same proof for existence of limit. 
    
To prove the second part of the claim, we plan to show, there exists a constant $C:= C(\beta,d)>0$ such that
\begin{equation}\label{eq:cov_lbd}
    \inf_{i \sim j} \text{Cov}(X_i,X_j) \ge C,
\end{equation}
where $i \sim j$ means $i$ and $j$-th vertices are neighbors in $\Lambda_{n,d}$. This, in turn will show that $\frac{1}{n}\sum_{i 
    \neq j} \text{Cov}^2(X_i,X_j) \ge \frac{C}{n} \sum_{i \ne j} \mathbf{1}_{i \sim j} =\Theta(1)$,
    since the number of neighbors on $\Lambda_{n,d}$ is $\Theta (1)$,
yielding the asymptotic distribution of $\sqrt{n}T_n$. \par

We now turn our attention to proving \eqref{eq:cov_lbd}. To this end, we invoke  the \textit{Edwards-Sokal coupling} between the measures $\bP_{\beta,\mathbf{Q}}$ and random-cluster model. The random cluster model, on a finite graph $G = (V,E)$ is a probability measure on $\{0,1\}^{E}$ defined by:
\begin{equation}
\phi_{G,p,q}(\omega) = p^{|\{e:\omega_e =1\}|} (1-p)^{|\{e:\omega_e =0\}|}q^{k(\omega)} \qquad \omega \in \{0,1\}^{E}.
\end{equation}
where $k(\omega)$ denotes the number of connected components of $\omega$ (where we think of $\omega$ as the graph $(V, \{e:\omega_e=1\})$). The Edwards-Sokal coupling (cf. \cite[Theorem 4.91]{grimmett2006random}) states that 
\begin{equation}\label{eq:edwards_sokal}
    \text{Cov}(X_i, X_j)= \phi_{G,p,2}(i \leftrightarrow j),
\end{equation}
where the notation $i \leftrightarrow j$ implies that the vertices $i$, $j$ are connected in the random-cluster model via open edges and $p:= 1- e^{2 \beta}$. Now, we invoke \cite[(3.23)]{grimmett2006random} with $p_1=p$, $p_2= p/2$, $q_1=2$ and $q_2=1$. Hence \begin{equation}
    \phi_{G,p,2}(i \leftrightarrow j) \ge \phi_{G,p/2,1}(i \leftrightarrow j) =p/2,
\end{equation}
if $i \sim j$ in $G$ since the random-cluster model $\phi_{G,p,q}$ with $q=1$ corresponds to Bernoulli percolation with each edge being open independently with probability $p$. Hence, if $i \sim j$, we have, by \eqref{eq:edwards_sokal}, $\text{Cov}(X_i, X_j) \ge \frac{1}{2} (1- e^{2 \beta})$, proving \eqref{eq:cov_lbd}. \par
Moreover, since $\overline{\mathbf{X}}_n \xrightarrow{\bP} 0$ as $\beta <\beta_c$, we have $\frac{1}{n} \sum_{i=1}^{n}(X_i- \overline\bX_n)^2= 1- \overline\bX^2_n \xrightarrow{\bP} 1$. Therefore, by Slutsky's theorem, we have $\sqrt{n} \rho_n \xrightarrow{d} N(0,v^2)$.\\
Finally, from \eqref{eq:vn_value}, we know for $\beta(p_1)< \beta(p_2)<\beta_c$, where $\beta{(p_j)}$ solves $p_j=\frac{1}{2}(1-e^{2\beta(p_j)})$, $j=1,2$,
\begin{align*}
    \lim\limits_{n \rightarrow \infty}v_n(\beta(p_1))&= 1+ \lim\limits_{n \rightarrow \infty}\frac{1}{n}\sum_{i 
    \neq j} \text{Cov}^2_{\beta(p_1),\bQ}(X_i,X_j) \\
    &\le 1+ \lim\limits_{n \rightarrow \infty}\frac{1}{n}\sum_{i 
    \neq j} \text{Cov}^2_{\beta(p_2),\bQ}(X_i,X_j) = \lim\limits_{n \rightarrow \infty}v_n(\beta(p_2)),
\end{align*}
where the inequality is due to \cite[Lemma 15]{bhattacharya2021sharp}. This shows that the asymptotic variance is a nondecreasing function of $\beta$, completing the proof.

\subsection{Proof of Theorem \ref{thm:cw_tn}}
\noindent    We want to show, for any $t \in \mathbb{R}$,
\begin{align}\label{eq:centered_clt}
    \lim\limits_{n \rightarrow \infty}\bE\Big(\exp(t\sqrt{n} T_n)\Big) = \exp \left(\frac{t^2 (1-m^2_1)(1-m^2_2) }{2}\right),
\end{align}
where, recall that, $T_n= \frac{1}{n} \sum_{i=1}^n(X_i -\overline{\bX}_n) Y_i$ and $m_1,m_2$ are defined in the statement of Theorem \ref{thm:cw_tn}. By Lemma \ref{lem:aux_var}, there exists $Z_2$ such that $Z_2 \sim N\Big(\overline{\mathbf{Y}}_n,\frac{1}{n \beta_2}\Big)$ and $Y_1, \ldots, Y_n$ are independent given $Z_2$, i.e., $\bP(\mathbf{Y}= \mathbf{y}|Z_2) \propto \prod_{i=1}^n e^{\beta_2 Z_2 y_i}$ where $\mathbf{y}= (y_1, \ldots y_n)$. Therefore, for any $i= 1,2, \ldots,n$, we have $\bE(Y_i|Z_2)= \tanh(\beta_2 Z_2)$ and $\text{Var}(Y_2|Z_2)= \sech^2(\beta_2 Z_2)$. \par 
Moreover, by Lemma \ref{lem:aux_var}, $Z_2 \xrightarrow{d} 0$ if $\beta_2 \le 1$ and $Z_2\xrightarrow{d} \frac{1}{2}(\delta_{m_2}+\delta_{-{m_2}})$ otherwise. Invoking continuous mapping theorem, we have
\begin{equation}\label{eq:secz_limit}
    \sech^2(\beta_2 Z_2) \stackrel{\bP}{\rightarrow}
    1- m^2_2,
\end{equation} 
where we used $\sech^2(\beta_2m_2)= 1- m^2_2$. Also, since $\sech^2(x) \le 1$, by the dominated convergence theorem, $\bE\Big(\sech^2(\beta_2 Z_2)\Big) \rightarrow 1- m^2_2$.

Using Lemma \ref{lem:cw_clt}, we obtain $\overline{\bX}_n \xrightarrow{d} 0$ if $\beta_1 \le 1$ and $\overline{\bX}_n\xrightarrow{d} \frac{1}{2}(\delta_{m_1}+\delta_{-{m_1}})$ otherwise (using symmetry of the model around $0$). Therefore, by continuous mapping theorem, 
\begin{equation}\label{eq:x_var_limit}
    \frac{1}{n}\sum_{i=1}^{n}(X_i-\overline{\bX}_n)^2 = 1- \overline{\bX}^2_n\stackrel{\bP}{\rightarrow} 1-m^2_1
\end{equation} 
Since $\overline{\bX}^2_n \le 1$, by the dominated convergence theorem, $\bE\Big(\frac{1}{n}\sum_{i=1}^{n}(X_i-\overline{\bX}_n)^2\Big) \rightarrow 1-m^2_1$. \par
Now, conditioned on $\bX,Z_2$,
\begin{align*}
    &\bE\Big(\exp(tT_n)|\bX,Z_2\Big)\\
    &=\exp\Big(\underbrace{\sum_{i=1}^{n} \Big(\log\cosh(\frac{t(X_i-\overline{\bX}_n)}{n^{1/2}}+\beta_2 Z_2)- \log\cosh(\beta_2Z_2)}_{=:W_n}\Big)\Big).
\end{align*}
By a Taylor series expansion,
\begin{align*}
    &W_n-\frac{t^2}{2}(1-m^2_1)(1-m^2_2)\\ &= \frac{t^2}{2} \left(\frac1n \sum_{i=1}^{n}(X_i-\overline{\bX}_n)^2 \sech^2(\beta_2 Z_2)- (1-m^2_1)(1-m^2_2)\right) + O\left(\frac{1}{\sqrt{n}}\right).
\end{align*}
Since $\sum_{i=1}^{n}(X_i-\overline{\bX}_n)^2 \xrightarrow{\bP} 1-m^2_1$ and $\sech^2(\beta_2 Z_2) \xrightarrow{d} 1-m^2_2$ by \eqref{eq:x_var_limit} and \eqref{eq:secz_limit}, we obtain
\begin{align}
    |W_n - \frac{t^2}{2}(1-m^2_1)(1-m^2_2)|=o_{\bP}(1).
\end{align}
To show \eqref{eq:centered_clt}, using uniform integrability, it is enough to show $$\limsup\limits_{n \rightarrow \infty}\bE(e^{2W_n- t^2 (1-m^2_1)(1-m^2_2)})<\infty,$$
which is immediate from the Taylor expansion above, since 
\begin{align*}
    \left|\frac1n \sum_{i=1}^{n}(X_i-\overline{\bX}_n)^2 \sech^2(\beta_2 Z_2)- (1-m^2_1)(1-m^2_2)\right| \le 4.
\end{align*}
This completes the proof of the Theorem.

\textbf{Proof of Corollary \ref{cor:cw_correlation}}
    When $\beta_1=\beta_2=0$, $X_i$, $Y_i$s are independent and the proof is immediate. For $\beta_i>0$ case, the conclusion follows combining Theorem \ref{thm:cw_tn} and \eqref{eq:x_var_limit} and using Slutsky's theorem.

\subsection{Proof of Theorem \ref{prop:er_var}}
Recall that $T_n$ is the sample covariance of $\mathbf{X}$ and $\mathbf{Y}$. We need the following auxilliary lemma.

\begin{lemma}\label{lem:iid_convergence}
    Suppose $\mathbf{Z}= (Z_1, \ldots, Z_n)$ are i.i.d. random variables with mean $0$ and variance $\tau^2_n$ with converging to $\tau^2 \in (0,\infty)$. Assume the matrix $\bQ_n$ satisfies Assumptions \ref{ass:reg}-\ref{ass:graphon_limit}. Set $B_n= \bQ_n - \frac{1}{n} \mathbf{1}\mathbf{1}^\top$. Suppose $\bX\sim \bP_{\beta_1,\bQ_n}$ and $\bX$ and $\mathbf{Z}$ are independent.
    
    
    Define $\sqrt{n} T^Z_n= \frac{1}{\sqrt{n}} \sum_{i=1}^n (X_i -\overline{\mathbf{X}}_n)Z_i$. Then,
    \begin{equation}
        (\sqrt{n} T^Z_n, \mathbf{Z}^\top B_n \mathbf{Z}) \xrightarrow{d} ((1-m^2_1)\tau_0 W_0, S_0),
    \end{equation}
    where $W_0 \sim N(0,1)$, $S_0$ is some random variable, and $W_0$ and $S_0$ independent, and $m_1$ is the unique positive solution of the equation $m=\tanh(\beta_1 m)$ if $\beta_1>1$ and $0$ otherwise.
\end{lemma}
We defer the proof of the above Lemma to Section \ref{sec:appendix}. We now invoke the Lemma \ref{lem:iid_convergence} to prove the following convergence result for Curie-Weiss model.
\begin{lemma}\label{lem:iid_to_cw}
    Suppose $\bX\sim \bP_{\beta_1, \bQ_n}$, $\mathbf{Y} \sim \bP_{\beta_2, \bQ^{\text{CW}}}$, $\bX,\mathbf{Y}$ are independent and $\bQ_n$ satisfies Assumptions \ref{ass:reg}-\ref{ass:graphon_limit}. Then 
    \begin{equation*}
        (\sqrt{n}T_n, \mathbf{Y}^\top B_n \mathbf{Y}) \xrightarrow{d} (W_1,S_0),
    \end{equation*}
    where $W_1\sim N(0,(1-m^2_1)(1-m^2_2))$, $S_0$ defined in Lemma \ref{lem:iid_convergence} and $W_1$ and $S_0$ are independent.
\end{lemma}
The proof of the Lemma is deferred to Section \ref{sec:appendix}. \par
Now, we are in a position to prove the Theorem \ref{prop:er_var}. To this end note the following observations:
\begin{enumerate}
\item The measures $P_{\beta_1,\bQ_n} \times P_{\beta_2,\bQ_n}$ and $P_{\beta_1,\bQ_n} \times P_{\beta_2,\bQ^{\text{CW}}}$ are mutually contiguous by \cite[Lemma 4.4 part (b)]{xu2022ising}.
    \item The log-likelihood ratio $$\log \frac{d(P_{\beta_1,\bQ_n} \times P_{\beta_2,\bQ_n})}{d(P_{\beta_1,\bQ_n} \times P_{\beta_2,\bQ^{\text{CW}}})}(\bX, \mathbf{Y})= \frac{\beta_2}{2} \mathbf{Y}^\top B_n \mathbf{Y} - Z_n (\beta_2, \bQ_n)+ Z_n(\beta_2,\bQ^{\text{CW}})$$ is a  function of $\mathbf{Y}^\top B_n \mathbf{Y}$.
    \item When $(\bX, \mathbf{Y}) \sim \bP_{\beta_1, \bQ^{\text{CW}}} \times \bP_{\beta_2, \bQ^{\text{CW}}}$, $\sqrt{n}T_n$ has a non-degenerate asymptotic distribution independent of $\mathbf{Y}^\top B_n \mathbf{Y}$.
\end{enumerate}
Therefore, by Le Cam's third Lemma \citep{van2000asymptotic}, the asymptotic distribution of $\sqrt{n}T_n$ is same under the measures $P_{\beta_1,\bQ_n} \times P_{\beta_2,\bQ_n}$ and $P_{\beta_1,\bQ_n} \times P_{\beta_2,\bQ^{\text{CW}}}$. This yields the asymptotic distribution by using Lemma \ref{lem:iid_to_cw}. \par
For the sample correlation, recall that $\frac{1}{n}\sum_{i=1}^{n}(X_i -\overline\bX^2_n)=1-\overline\bX^2_n \xrightarrow{\bP} 1-m^2_1$ and $\frac{1}{n}\sum_{i=1}^{n}(Y_i- \overline{\mathbf{Y}}^2_n) \xrightarrow{\bP} 1- m^2_2$. This, along with Slutsky's Theorem, completes the proof.

\subsection{Proof of Theorem \ref{lem:ggm}} 
\noindent (i) Here, we assume  $\bX,\mathbf{Y} \sim N(0, \Sigma)$ are independent and $T_n$ is their sample covariance. So, $nT_n= \bX^\top J \mathbf{Y}$ where $J= I- \frac{1}{n} \mathbf{1}\mathbf{1}^\top$. Conditioned on $\bX$, we know $nT_n \sim N(0, \bX^\top J \Sigma J \bX)$. Recall that $\widetilde{\Sigma}:= \Sigma^{1/2} J \Sigma^{1/2}$ and let $\widetilde{\Sigma}= P \tilde\Lambda P^\top$ be the singular value decomposition of $\widetilde{\Sigma}$. Therefore
\begin{equation}\label{eq:sigma_var}
    \bX^\top J \Sigma J \bX= \bZ^\top \widetilde\Sigma^2 \bZ= \bZ^\top P \tilde\Lambda^2 P^\top \bZ \stackrel{D}{=} \bZ^\top \tilde\Lambda^2 \bZ,
\end{equation}
where $\bZ \sim N(0,I_n)$. Now, we invoke Lemma \ref{lem:bilinear_concentration} since, $\tilde\lambda^2_1 \ll \sum_{i=1}^{n} \tilde\lambda^2_i$, to obtain $\frac{\bZ^\top \tilde\Lambda^2 \bZ}{\bE(\bZ^\top \tilde\Lambda^2 \bZ)} \xrightarrow{\bP} 1$. Now, $\bE(\bZ^\top \tilde\Lambda^2 \bZ)= \text{Trace}(\tilde\Lambda^2)= \sum_{i=1}^{n} \tilde\lambda^2_i$. Hence, by Slutsky's Theorem, we prove \eqref{eq:small_eval}. \par
To show asymptotic distribution of sample correlation $\rho_n$, note that, $\tilde\lambda_1 \ll \sum_{i=1}^{n} \tilde\lambda_i$, hence by Lemma \ref{lem:bilinear_concentration}, we have $\frac{\bX^\top J \bX}{\bE(\bX^\top J \bX)} \xrightarrow{\bP}1$. Also, $\bE(\bX^\top J \bX)= \sum_{i=1}^{n} \tilde\lambda_i$. Hence,
\begin{align*}
    \sqrt{n}\rho_n a_n &= \sqrt{n} \frac{n T_n a_n}{\sqrt{(\bX^\top J \bX)(\bY^\top J \bY)}} = (1+ o_{\bP}(1)) \sqrt{n} \frac{ n T_n a_n}{ \sum_{i=1}^{n}\tilde\lambda_i}\\
    &= (1+ o_{\bP}(1) \frac{nT_n}{\sqrt{\sum_{i=1}^{n} \tilde\lambda^2_i}} \xrightarrow{d} N(0,1).
\end{align*}
Also, by Cauchy-Schwarz inequality and the fact $\tilde \lambda_i \ge 0$, we have $a_n \le 1$. Therefore $\liminf_{n \rightarrow \infty} \text{Var}(\sqrt{n} \rho_n) \ge \lim\limits_{n \rightarrow \infty} \frac{1}{a_n} \ge 1$. 

(ii) Now, we assume $\tilde{\lambda}^2_1 \gg n \tilde{\lambda}^2_2$. As the previous part, conditioned on $\bX$, $nT_n \sim N(0, \bX^\top J \Sigma J \bX)$. Hence, $\frac{nT_n}{\sqrt{\bX^\top J \Sigma J \bX}}\sim N(0,1)$ and it is independent of $\bX$. So, it is enough to show $\zeta_n:= \frac{\bX^\top J \Sigma J \bX}{\tilde\lambda^2_1} \xrightarrow{d} \chi^2_1$. By \eqref{eq:sigma_var}, we have 

\begin{align*}
    \zeta_n \stackrel{d}{=} \frac{\sum_{i=1}^n \tilde\lambda^2_i Z^2_i}{ \tilde\lambda^2_1}= Z^2_1 + \sum_{i=2}^n \frac{\tilde\lambda^2_i}{\tilde\lambda^2_1}Z^2_i &\le Z^2_1 + \frac{n \tilde\lambda^2_2}{\tilde\lambda^2_1} \Big(\frac{1}{n} \sum_{i=2}^{n} Z^2_i\Big)\\
    & =Z^2_1+ o_{\bP}(1),
\end{align*}
since $\tilde{\lambda}^2_1 \gg n \tilde{\lambda}^2_2$ and $\sum_{i=2}^{n} Z^2_i= O_{\bP}(n)$. As $Z^2_1 \sim \chi^2_1$, we get the desired result.

To compute the asymptotic distribution of $\rho_n$, note that, following a similar calculation as above, if $Z_1,Z_2 \sim N(0,I)$ are independent, then
\begin{align*}
    \rho_n= \frac{\sum_{i=1}\tilde\lambda_i Z_{1i}Z_{2i}}{\sqrt{\Big(\sum_{i=1}\tilde\lambda_i Z^2_{1i}\Big)\Big(\sum_{i=1}\tilde\lambda_i Z^2_{2i}\Big)}}= \frac{Z_{11}Z_{21}+o_{\bP}(1)}{\sqrt{Z^2_{11}Z^2_{21}+o_{\bP}(1)}} \xrightarrow{d} \frac{1}{2} (\delta_{1}+\delta_{-1}),
\end{align*}
concluding the proof of the Theorem.

\textbf{Proof of Lemma \ref{lem:contagion}}
Here we assume the Contagion model, i.e., $\bX,\mathbf{Y} \sim N(0, \Sigma)$ are independent and $\Sigma= (1-\rho)I + \rho \mathbf{1} \mathbf{1}^\top$, $0<\rho<1$.

\noindent (i) We want to compute the asymptotic distribution of $\sqrt{n}T_n$. To this end, note that $J$ has eigenvalues $1$ with multiplicity $(n-1)$ and $0$ with multiplicity $1$ and $J\Sigma= \Sigma J= (1-\rho)J$. Since the eigenvalues of $(\Sigma^{1/2})^\top J \Sigma^{1/2}$ are same as the eigenvalues of $J\Sigma$, we have, using the notation of Theorem \ref{lem:ggm}, $\tilde\lambda_1= \ldots= \tilde\lambda_{n-1} = (1-\rho)$ and $\tilde\lambda_n=0$. Therefore,  $\tilde{\lambda}^2_1 \ll \sum_{i=1}^{n}\tilde{\lambda}^2_i$ and we invoke Theorem \ref{lem:ggm} (i) to claim 
$ \frac{nT_n}{\sqrt{(n-1)(1-\rho)^2}}\xrightarrow{d} N(0,1)$, i.e., $\sqrt{n}T_n \xrightarrow{d} N(0, (1-\rho)^2)$. Moreover,
\begin{align*}
    a_n = \frac{(1-\rho) (n-1)}{\sqrt{n (n-1) (1-\rho)^2}}\rightarrow 1,
\end{align*}
yielding $\sqrt{n} \rho_n \xrightarrow{d} N(0,1)$.\\
(ii) We want to compute the asymptotic distribution of $S_n= \bX^\top \bY$. Borrowing notation from Theorem \ref{lem:ggm}, we have $\lambda_1= 1+(n-1)\rho$ and $\lambda_2 =\ldots= \lambda_n= (1-\rho)$. Since $\lambda^2_1 \gg n \lambda^2_2$, we invoke Theorem \ref{lem:ggm} to obtain 
$\frac{n S_n}{1+ (n-1) \rho} = (1+o(1)) \frac{S_n}{\rho}\xrightarrow{d} AB$, where $A\sim N(0,1)$, $B\sim \sqrt{\chi^2_1}$ and $A,B$ are indendent. This completes the proof.

\subsection{Proof of Proposition \ref{thm:ols}}
We need to compare the naive variance with the variance of  $\hat{\beta}_{\rm ols}$. 
To this end, note the following conclusions:
\begin{align}\label{eq:prob_conv_ols}
    \frac{1}{n} \|\varepsilon\|^2\stackrel{\bP}{\rightarrow} \int g(x) dx, \qquad \frac{1}{n} \|\bX\|^2\stackrel{\bP}{\rightarrow} \int f(x) dx,\qquad \frac{1}{n} \varepsilon^\top \bX \stackrel{\bP}{\rightarrow} 0.
\end{align}
The first convergence in \eqref{eq:prob_conv_ols} is true since $\bE( \frac{1}{n} \|\varepsilon\|^2)= \frac{1}{n} \sum_{i=1}^{n} g(i/n) \rightarrow \int g(x) dx$, and $\text{Var}( \frac{1}{n} \|\varepsilon\|^2)\rightarrow0$. The other two convergences follow similarly. Moreover
\begin{equation}\label{eq:prob_ols_ii}
    \frac{1}{n}\bX^T\Sigma_\varepsilon\bX \stackrel{\bP}{\rightarrow} \int f(x)g(x)dx,
\end{equation}
since $\bE(\frac{1}{n}\bX^T\Sigma_\varepsilon\bX)= \frac{1}{n} \sum_{i=1}^n f(i/n)g(i/n) {\rightarrow} \int f(x)g(x)dx$, and $\text{Var}(\frac{1}{n}\bX^T\Sigma_\varepsilon\bX) \rightarrow 0$.

Hence,

\begin{align*}
    n \widehat{\text{Var}_{\text{naive}}}&= \frac{\boldsymbol{\varepsilon}^T(\|\bX\|^2- \bX\bX^\top)\boldsymbol{\varepsilon}}{ \|\bX\|^4}= \frac{\|\varepsilon\|^2}{\|\bX\|^2}\left(1- \frac{(\varepsilon^\top\bX)^2}{\|\varepsilon\|^2 \|\bX\|^2}\right) \stackrel{\bP}{\rightarrow} \frac{\int g(x) dx}{\int f(x) dx},
\end{align*}
where the first convergence follows from \eqref{eq:prob_conv_ols}. Further, we have,
\begin{align*}
    n \text{Var}(\hat{\beta}_{\rm ols})= \frac{n \bX^T\Sigma_\varepsilon\bX}{ \|\bX\|^4}\stackrel{\bP}{\rightarrow} \frac{\int f(x)g(x) dx}{(\int f(x) dx)^2}
\end{align*}
using \eqref{eq:prob_conv_ols} and \eqref{eq:prob_ols_ii}. Naive inference is valid iff 
$\frac{\text{Var}(\hat{\beta}_{\rm ols})}{\widehat{\text{Var}_{\text{naive}}}}< 1$, which by the above two displays equivalent to 
$$\int f(x) g(x) dx \leq \Big(\int f(x) dx\Big) \Big(\int g(x) dx\Big),$$
concluding our proof.

\section{Conclusion}
This paper is a first exploration of the conditions under which nonsense associations -- or the opposite phenomenon of valid / conservative inference -- occur for dependent data with dependence governed by MRF distributions. Our results serve as a proof of concept that researchers ought to be concerned about nonsense associations in settings beyond time series, but more work is needed to fully characterize the effects of unit-level dependence on measures of association. In particular, we report several surprising findings, and the extent to which these can occur in other settings should be explored in future work: (1) settings where the asymptotic behavior of the sample correlation is qualitatively different from the asymptotic behavior of the sample covariance, and (2) settings where unit-level positive dependence results in variance deflation and conservative inference. The latter is especially surprising because it contradicts the notion that positive dependence results in smaller effective sample sizes for statistical inference. Finally, our results on OLS upend the way that dependent residuals is typically presented in introductory regression texts.

\section*{Appendix}\label{sec:appendix}
\noindent Here, we state the technical Lemmas along with proofs which we have used throughout the paper. The following lemma follows from \cite[Lemma 3]{mukherjee2018global}.

\begin{lemma}\label{lem:aux_var}
    Suppose $\bX\sim \bP_{\beta,\bQ^{\rm CW}}$ with $\bQ^{\rm CW}_{i,j}=\mathbf{1}(i\neq j)/n$. Define a random variable $Z_n$ which given $\bX$ has a distribution $ N(\overline{\mathbf{X}}_n,\frac{1}{n\beta})$. Then under $\bP_{\beta,\bQ^{\rm CW}}$, given $Z_n=z_n$, each $X_i$'s are i.i.d. with 
	\begin{align*}
	\bP_{\beta,\bQ^{\rm CW}}(X_i=x_i|Z_n=z_n)=\frac{e^{\beta z_n x_i}}{e^{\beta z_n}+e^{-\beta z_n}}.
	\end{align*}
 Moreover, $Z \xrightarrow{d} 0$ if $\beta \le 1$ and $Z\xrightarrow{d} \frac{1}{2}(\delta_{m}+\delta_{-{m}})$ otherwise, where $m= m(\beta)$ is the unique positive root of $m=\tanh(\beta m)$. 
\end{lemma}

\noindent The following lemma follows from \cite[Theorems 1-3]{ellis1978statistics}. 
\begin{lemma}\label{lem:cw_clt} Suppose $\bX\sim \bP_{\beta,\bQ^{\rm CW}}$ with $\bQ^{\rm CW}_{i,j}=\mathbf{1}(i\neq j)/n$.
\begin{enumerate}
\item[(a)]
If $\beta\in (0,1)$ then $$\sqrt{n}\overline{\bX}_n\stackrel{d}{\rightarrow}N\left(0,\frac{1}{1-\beta}\right).$$

\item[(b)]
If $\beta>1$ then we have $$\sqrt{n}(\overline{\bX}_n-m|\overline{\bX}_n>0)\stackrel{d}{\rightarrow}N\left(0,\frac{1-m^2}{1-\beta(1-m^2)}\right),$$
where $m$ is the unique positive root of the equation $t=\tanh(\beta t)$.

\item[(c)]
If $\beta=1$ then we have $$n^{1/4}\overline{\bX}_n\stackrel{d}{\rightarrow}Y,$$
where $Y$ is a random variable with density proportional to $e^{-y^4/12}$.
\end{enumerate}
\end{lemma}

Next, we prove a lemma regarding concentration inequalities of quadratic forms of multivariate normal distributions.
\begin{lemma}\label{lem:bilinear_concentration}
    Suppose $A$ is an $n \times n$ matrix with eigenvalues $\lambda_1 \ge \lambda_2 \ge \ldots \ge 0$ and $\frac{\lambda_1}{\sum_{i=1}^{n}\lambda_i} = o(1)$. If $Z \sim N(0,I_n)$, then $\frac{Z^\top A Z}{\bE(Z^\top A Z)} \stackrel{\bP}{\rightarrow}1$.
\end{lemma}
\begin{proof}
    By Markov's inequality, for any $t>0$,
    \begin{align}\label{eq:markov_bilin}
        \bP\Bigg(\Big\vert\frac{Z^\top A Z}{\bE(Z^\top A Z)}-1\Big\vert \ge t\Bigg) \le \frac{\text{Var}(Z^\top A Z)}{t^2 (\bE(Z^\top A Z))^2}
    \end{align}
    Note that, $\bE(Z^\top A Z) = \text{Tr}(A)= \sum_{i=1}^{n}\lambda_i$ and $\text{Var}(Z^\top A Z)= 2\text{Tr}(A^2) =2\sum_{i=1}^n \lambda^2_i$. So, the RHS of \eqref{eq:markov_bilin} is
    \begin{align*}
        \frac{2\sum_{i=1}^n \lambda^2_i}{t^2 (\sum_{i=1}^{n}\lambda_i)^2 } \le \frac{2\lambda_1}{ t^2\sum_{i=1}^{n}\lambda_i} =o(1),
    \end{align*}
    by our assumption. This completes the proof.
    \end{proof}

Next, we provide the proof of Lemmas omitted before.

\textbf{Proof of Lemma \ref{lem:iid_convergence}}
    The proof is done in two steps.
    \begin{enumerate}
        \item Suppose $G_1, \ldots G_n$ are i.i.d. $N(0,\tau^2_n)$ random variables. Then the conclusion of the Lemma holds.
        \item For any $a,b \in \mathbb{N}$, we have
        $$\bE(\sqrt{n} T^Z_n)^a(\mathbf{Z}^\top B_n \mathbf{Z})^b - \bE(\sqrt{n} T^G_n)^a(\mathbf{G}^\top B_n \mathbf{G})^b \rightarrow 0,$$
        where  $\sqrt{n} T^G_n= \frac{1}{\sqrt{n}} \sum_{i=1}^n (X_i -\overline{\mathbf{X}}_n)G_i$
    \end{enumerate}
    To show $(1)$, we will throughout assume $\tau^2_n=1$ and the case of general $\tau^2_n$ is immediate by change of scale. We begin by noting that for $a,b \in \mathbb{R}$ with $ b<1/2$, one has \begin{equation}\label{eq:normal_mgf}
        \bE(\exp(aG_1+bG^2_1))= \frac{1}{\sqrt{1-2b}} \exp(\frac{a^2}{2(1-2b)}).
    \end{equation} 
    Define $\upsilon_n= \frac{1}{\sqrt{n}}(X_1-\overline{\mathbf{X}}_n, \ldots X_n- \overline{\mathbf{X}}_n)$ so that $\sqrt{n} T^G_n=\upsilon^T_n \mathbf{G}$. Let $\bQ_n= P^\top \Lambda P = \sum_{i=1}^n \lambda_{i,n} \mathbf{p}_i\mathbf{p}^\top_i$ be the spectral decomposition of $\bQ_n$, where the eigenvalues $\lambda_{i,n}$, $i=1, 2, \ldots, n$ be the eigenvalues of $\bQ_n$ sorted in the decreasing order of absolute values. Since $\sum_{j=1}^{n} Q_{n,ij}=1$ for all $i$, we have $\lambda_{1,n}=1$ and $\mathbf{p}_1= \frac{1}{\sqrt{n}}\mathbf{1}$. Setting $\tilde{\mathbf{G}}= P\mathbf{G}$, where $\mathbf{G}=(G_1,\ldots, G_n)$, we have, $$\Big(\upsilon^T_n \mathbf{G}, \mathbf{G}^\top B_n \mathbf{G} \Big)\stackrel{d}{=} \Big(\upsilon^T_n P^T\tilde{\mathbf{G}}, \sum_{i=2}^n \lambda_{i,n} \tilde{G}^2_i\Big) \stackrel{d}{=} \Big(\mathbf{a}^\top_n\mathbf{G}, \sum_{i=2}^n \lambda_{i,n}G^2_i \Big),$$
where we defined $\mathbf{a}_n= P \upsilon_n$ and used the fact $P\mathbf{G}\stackrel{d}{=} \mathbf{G}$. Now, invoking \eqref{eq:normal_mgf}, we have, for $s,t\in \mathbb{R}$ and $t < \frac{1}{2}$ and fixed $\bX$,
\begin{align}\label{eq:mgf_split}
    \log \bE \Big(\exp(s\mathbf{a}^\top_n\mathbf{G}+t \sum_{i=2}^n \lambda_{i,n}G^2_i)\Big) &= -\frac{1}{2}\sum_{i=2}^n \log (1-2 t \lambda_{i,n})+ \frac{s^2a^2_1}{2} \\ &+ \sum_{i=2}^{n}\frac{s^2 a^2_i}{2(1-2b_i)}.
\end{align}
Denote the first summand of RHS by $f_n(t)$. Following the proof of \cite[Lemma 4.1]{xu2022ising}, we have $f_n(t)$ converges to the log of moment generating function of a random variable $S_0$. Next we analyse the third summand of the above display. To this end, we have 
\begin{align}\label{eq:dense_taylor}
    &\sum_{i=2}^{n}\frac{s^2 a^2_i}{2(1-2t \lambda_{i,n})}= \sum_{i=2}^{n}\sum_{j=0}^{\infty}\frac{s^2 a^2_i}{2} (2t \lambda_{i,n}))^j \nonumber \\
    &= \sum_{i=2}^{n} \frac{s^2 a^2_i}{2} + \underbrace{2ts^2 \sum_{i=2}^n a^2_i \lambda_{i,n} + \sum_{j=2}^{\infty}\sum_{i=2}^{n}\frac{s^2 a^2_i}{2} (2t \lambda_{i,n}))^j}_{R_n}
\end{align}
    To control the second summand above, $\mathbf{1}^T \upsilon_n=0$ yielding $a_1=0$. Hence,
    \begin{align*}
       \sum_{i=2}^n a^2_i \lambda_{i,n}= \sum_{i=1}^n a^2_i \lambda_{i,n} &=\upsilon^\top_n \bQ_n \upsilon_n \\ & = \frac{1}{n} \mathbf{X}^\top (I- \frac{1}{n}\mathbf{1}\mathbf{1}^\top) \bQ_n (I- \frac{1}{n}\mathbf{1}\mathbf{1}^\top)\mathbf{X}\\
& = \frac{1}{n}\mathbf{X}^\top B_n \mathbf{X} \xrightarrow{\bP} 0,
    \end{align*}
since $\bQ_n \mathbf{1}= \mathbf{1}$ and $\mathbf{X}^\top B_n \mathbf{X}$ is tight by \cite[Lemma 3.1]{xu2022ising} (note that, \cite[Lemma 3.1]{xu2022ising} states the result conditional on $\overline{\mathbf{X}}>0$, but the same proof works for the distribution conditional on $\overline{\mathbf{X}} \le 0$). This yields that the second summand of \eqref{eq:dense_taylor} is $o_{\bP}(1)$. Finally the third summand of \eqref{eq:dense_taylor} by 
\begin{align*}
    \sum_{j=2}^{\infty}\sum_{i=2}^{n}\frac{s^2 a^2_i}{2} (2t \lambda_{i,n})^j&= \sum_{j=2}^{\infty} \frac{s^2 (-2t)^j}{2} \sum_{i=2}^{n} a^2_i\lambda^j_{i,n} \\
    &\le \sum_{j=2}^{\infty} \frac{s^2 (2t)^j}{2} \sum_{i=2}^{n} a^2_i\lambda^2_{i,n} \le \frac{s^2}{2 (1-2t)} \sum_{i=2}^{n} a^2_i\lambda^2_{i,n},
\end{align*}
where the first inequality uses $|\lambda_{i,n}|\le 1$. Note that,
\begin{equation*}
    \sum_{i=2}^{n} a^2_i\lambda^2_{i,n}= \upsilon^\top Q^2_n \upsilon  =\frac{1}{n} \mathbf{X}^\top B^2_n \mathbf{X} \xrightarrow{\bP} 0
\end{equation*}
invoking \cite[Lemma 3.1]{xu2022ising} as above. Therefore, we have, by \eqref{eq:mgf_split},
\begin{equation*}
    \log \bE \Big(\exp(s\mathbf{a}^\top_n\mathbf{G}+t \sum_{i=2}^n \lambda_{i,n}G^2_i)\Big)= f_n(t) +\frac{s^2}{2} \mathbf{a}^\top_n \mathbf{a}_n+ R_n,
\end{equation*}
where $R_n= o_{\bP}(1)$ and 
$|R_n| \le 2|t|s^2 + \frac{s^2}{2|1-2t|}$. Moreover, $\mathbf{a}^\top_n \mathbf{a}_n= \upsilon^T_n \upsilon_n= \frac{1}{n} \sum_{i=1}^{n}(X_i - \overline{\mathbf{X}}_n)^2 \rightarrow 1-m^2_1$ and $|\mathbf{a}^\top_n \mathbf{a}_n| \le 4$. Therefore, by the dominated convergence theorem, we have 
\begin{align*}
    \lim\limits_{n\rightarrow \infty}\bE \Big(\exp(s\mathbf{a}^\top_n\mathbf{G}+t \sum_{i=2}^n \lambda_{i,n}G^2_i)\Big) = e^{\frac{s^2}{2}(1-m^2_1)}\bE e^{tS_0}.
\end{align*}
 
 This shows, $(\upsilon^T_n \mathbf{G}, \mathbf{G}^\top B_n \mathbf{G}) \xrightarrow{d} (S_0, \sqrt{1-m^2_1}W_0)$ where $W_0 \sim N(0,1)$ and $S_0, W_0$ are asymptotically independent proving (1).

    To show $(2)$, note that, $|B_n(i,j)| \le \frac{4}{n}$, since $\bQ_n$ is the (scaled) adjacency matrix of $\mathcal{G}_{n,p}$ or $\mathcal{G}^d_n$ and $|c_{i,n}| \le \frac{\kappa}{\sqrt{n}}$ by our assumption. Set $r= \frac{a}{2}+b$ and $S(l,2r)$ be the set of all positive integer solutions to the equation $\sum_{i=1}^l \alpha_i=2r$. Then we have,
\begin{align*}
    &|\bE(\sqrt{n} T^Z_n)^a(\mathbf{Z}^\top B_n \mathbf{Z})^b- \bE(\sqrt{n} T^G_n)^a(\mathbf{G}^\top B_n \mathbf{G})^b| \\
    & \le n^{-r} \kappa^a 4^b \sum_{l=1}^{2r} n^l \sum_{\mathbb{\alpha} \in S(l,2r)}|\bE \prod\limits_{i=1}^l Z^l_i-\bE \prod\limits_{i=1}^l G^l_i|,
\end{align*}
and the RHS in the above display can be bounded exactly as the proof of \cite[Lemma 4.1]{xu2022ising}.

\textbf{Proof of Lemma \ref{lem:iid_to_cw}}
We do the entire proof conditioning on $\bX$. By Lemma \ref{lem:aux_var}, we have the existence of auxilliary variable $Z_{n,2}$ such that $Y_1, \ldots, Y_n$ are independent given $Z_{n,2}$ with

\begin{align*}
    \bE(Y_i|Z_{n,2})=\tanh(\beta_2 Z_{n,2})=: \mu, \quad \text{Var}(Y_i|Z_{n,2})= \sech^2(\beta_2 Z_{n,2}), \text{ for all }i.
\end{align*}
Define $\bmu_n= \mu \mathbf{1}$ and recall from proof of Lemma \ref{lem:iid_convergence}, we have  $\upsilon_n= \frac{1}{\sqrt{n}}(X_1-\overline{\mathbf{X}}_n, \ldots X_n- \overline{\mathbf{X}}_n)$. Therefore, we have, using $B_n \mathbf{1}=0$, given $Z_n$
\begin{align*}
    (\sqrt{n}T_n, \mathbf{Y}^\top B_n \mathbf{Y})&= (\upsilon^\top_n \mathbf{Y},\mathbf{Y}^\top B_n \mathbf{Y}) \\ 
    &= (\upsilon^\top_n (\mathbf{Y} -\bmu_n), (\mathbf{Y} -\bmu_n)^\top B_n (\mathbf{Y} -\bmu_n))\\
    & \xrightarrow{d} (\sqrt{(1-m^2_1)(1-m^2_2)}W_0,S_0),
\end{align*}
 with $W_0 \sim N(0,1)$, $S_0$ as in Lemma \ref{lem:iid_convergence} and $W_0, S_0$ being independent by Lemma \ref{lem:iid_convergence} since $\sech^2(\beta_2Z_{n,2}) \xrightarrow{\bP} 1-m^2_2$ by Lemma \ref{lem:aux_var}.

\bibliographystyle{plain}
\bibliography{ref}

\end{document}